\documentclass[a4paper,12pt]{amsart}
\usepackage[utf8]{inputenc}
\usepackage{amsmath}
\usepackage{amsfonts}
\usepackage{amsthm}
\usepackage{amssymb}
\usepackage[mathscr]{euscript}
\usepackage{xcolor}
\usepackage{enumerate}
\usepackage{tikz}
\usetikzlibrary{positioning}
\usepackage{amstext}
\usepackage{enumerate}
\usepackage{amsfonts}
\usepackage{textcomp}
\usepackage{amssymb}
\usepackage{setspace}
\usepackage[all]{xy}
\usepackage{color}

\usepackage[
margin=1in,
marginpar=2cm,
includefoot,
footskip=30pt,
]{geometry}

\newtheorem{theorem}[equation]{Theorem}
\newtheorem{lemma}[equation]{Lemma}
\newtheorem{proposition}[equation]{Proposition}
\newtheorem{corollary}[equation]{Corollary}

\theoremstyle{definition}
\newtheorem{definition}[equation]{Definition}

\theoremstyle{remark}
\newtheorem{example}[equation]{Example}
\newtheorem{remark}[equation]{Remark}

\numberwithin{equation}{section}

\newcommand{\osf}{{\normalfont \textsf{X}}}

\newcommand{\lang}{\CL_{\osf}}

\newcommand{\ualgshift}{\TCA_R(\osf)}
\newcommand{\uCalgshift}{\TCA_{\mathbb C}(\osf)}

\newcommand{\udalgshift}{\TCD_R(\osf)}

\newcommand{\ucsalgshift}{\TCO_\osf}

\newcommand{\ucalgshift}{\TCO_{\osf}}

\newcommand{\alf}{\mathscr{A}}

\newcommand{\N}{\mathbb{N}}
\newcommand{\Z}{\mathbb{Z}}
\newcommand{\F}{\mathbb{F}}

\newcommand{\CA}{\mathcal{A}}

\newcommand{\CD}{\mathcal{D}}

\newcommand{\CL}{\mathcal{L}}
\newcommand{\CO}{\mathcal{O}}

\newcommand{\TCA}{\widetilde{\CA}}
\newcommand{\TCB}{\mathcal{U}}
\newcommand{\TCD}{\widetilde{\CD}}
\newcommand{\TCO}{\widetilde{\CO}}

\newcommand{\nn}{\mathbb{N}}

\newcommand{\eword}{\omega}

\newcommand{\vecspan}{\operatorname{span}}

\title[The socle of subshift algebras]{The socle of subshift algebras, with applications to subshift conjugacy}

\author[D. Gonçalves]{Daniel Gonçalves}
\address[Daniel Gonçalves and Danilo Royer]{Departamento de Matem\'atica, Universidade Federal de Santa Catarina, 88040-970 Florian\'opolis SC, Brazil. }
\email{daemig@gmail.com \\ daniloroyer@gmail.com}
\author[D. Royer]{Danilo Royer}

\begin{document}



\thanks{The first author was partially supported by Capes-Print Brazil, Conselho Nacional de Desenvolvimento Cient\'ifico e Tecnol\'ogico (CNPq) - Brazil, and Funda\c{c}\~ao de Amparo \`a Pesquisa e Inova\c{c}\~ao do Estado de Santa Catarina (FAPESC).}

\begin{abstract} 

We introduce the concept of "irrational paths" for a given subshift and use it to characterize all minimal left ideals in the associated unital subshift algebra. Consequently, we characterize the socle as the sum of the ideals generated by irrational paths. Proceeding, we construct a graph such that the Leavitt path algebra of this graph is graded isomorphic to the socle. This realization allows us to show that the graded structure of the socle serves as an invariant for the conjugacy of Ott-Tomforde-Willis subshifts and for the isometric conjugacy of subshifts constructed with the product topology. Additionally, we establish that the socle of the unital subshift algebra is contained in the socle of the corresponding unital subshift C*-algebra.

\end{abstract}

\maketitle

\section{Introduction}

Leavitt path algebras are noncommutative algebras that have attracted significant attention, due to their versatility and deep connections with various branches of mathematics, including combinatorics, C*-algebras, and symbolic dynamics. Particularly noteworthy, the classification of Leavitt path algebras is closely related to the classification of graph C*-algebras and those of subshifts of finite type. This relation enables the utilization of algebraic techniques and the inherent structure of Leavitt path algebras to address problems in symbolic dynamics, see \cite{Wconjec} and \cite{WillieHazrat} for a recent overview.

Expanding from the setup above, C*-algebras associated with subshifts over finite alphabets have been formally defined in prior works. However, due to the intricacies involved in dealing with arbitrary subshifts (not limited to finite type), the definition has gone through refinements, culminating with the formulation given in \cite{CarlsenShift}. Among other applications, conjugacy of subshifts can be described in terms of certain isomorphisms of the associated algebras, see \cite{BrixCarlsen}.

In the general context of subshifts over arbitrary alphabets, both purely algebraic and C*-algebraic subshift algebras have recently been defined, see \cite{BGGV3, BGGV}. These algebras generalize Leavitt path algebras of graphs and ultragraphs and can be seen as Leavitt path algebras of labeled graphs, \cite{BP1, BCGWLabel}. As with Leavitt path algebras, these new algebras exhibit a strong correlation with the underlying dynamics of the associated subshifts, including the characterization of conjugacy between Ott-Tomforde and Willis subshifts via algebraic terms. Consequently, the understanding of their structure is important. 

When addressing infinite alphabets, multiple definitions of a subshift exist. Among these, the approach outlined by Ott-Tomforde and Willis is extensively studied, see for example \cite{DDStep, MR3600124, GSS, GSS2s, OTW}, particularly in relation to the noncommutative subshift algebras introduced in \cite{BGGV3, BGGV}, exhibiting a strong connection with these algebras. It is worth noting that the subshifts introduced by Ott-Tomforde and Willis, hereafter referred to as OTW-subshifts, conform to the typical definition of subshifts when the alphabet is finite.

Our contribution, in this paper, to the description of the structure of subshift algebras lies in the socle and its grading. In the context of Leavitt path algebras, the socle is examined in \cite{Socle, Socle1}, and used in the classification program to generate an atlas of Leavitt path algebras of small graphs in \cite{atlas}. The socle series is studied in \cite{MR2823984} 
and, additionally, the socle of Kumijian-Pask algebras is considered in \cite{MR3354056}. In our investigation of the socle of a purely algebraic subshift algebra, we have identified a graph wherein the socle of the subshift algebra is graded isomorphic to the Leavitt path algebra of the said graph. This finding is novel even within the context of Leavitt path algebras, as previous studies have characterized the socle solely as an ideal, and enables us to discern properties of the socle and use it in applications. Indeed, it allows us to establish a combinatorial criterion that the underlying graphs of conjugate OTW-subshifts must satisfy. Building on \cite[Theorem~6.11]{BGGV3}, this implies that the graded socle serves as an invariant for isometric conjugacy of subshifts with the $\frac{1}{2^i}$-metric (which induces the product topology, see the discussion above \cite[Theorem~6.11]{BGGV3} for the precise definition of such subshifts). We give more details of our work below.

We devote Section~2 to preliminaries, where we recall the notions of subshift, the unital algebra associated with a subshift, and a few results that we use in the paper. We refer the reader to \cite{BGGV} for a more comprehensive introduction to subshift algebras. 

In Section~3, we describe the minimal left ideals in an unital subshift algebra, see Corollary~\ref{laVuelta}. To do this, we identify certain special elements in a subshift, which we call irrational paths (Definition~\ref{chuvacalor}). We show that any minimal left ideal is isomorphic to a left ideal generated by a one-point set given by an irrational path. Previously, left minimal ideals were linked with line paths, see \cite{Irreduciblerepresentations}. The use of irrational paths (which contain the line paths) completes the description of the left minimal ideals. 

With the description of the minimal left ideals obtained in Section~3 we proceed, in Section~4, to the study of the socle, which is the sum of all left minimal ideals of the algebra. We prove that the socle of an unital subshift algebra is not only isomorphic but equal, to the sum of left ideals generated by the irrational paths, see Theorem~\ref{socle description}. This result is new also in the context of Leavitt path algebras, where it was previously shown that the socle contains the ideal generated by line points, see Remark~\ref{carregador}. As a consequence of our results, we show that the socle of a unital subshift algebra is always different from the algebra, and is nonzero if and only if the set of irrational points is non-empty, see \ref{iogurtecaseiro}. We finish the section describing a smaller generating set for the socle. For this, we consider tail equivalence (Definition~\ref{def equiv line paths}) in the set of irrational paths and show that the socle is the direct sum, of the left minimal ideals generated by each equivalence class, see Corollary~\ref{socle characterization}. 

The construction of a graph associated with a subshift is the focus of Section 5. We show that the Leavitt path algebra of this graph is graded isomorphic to the subshift algebra, see Corollary~\ref{doubleskiff}. Moreover, we observe that the aforementioned graph is always acyclic, row-finite, and sinkless, see Remark~\ref{rmk48}. We use this to characterize, in terms of Condition~(Y), when the socle is strongly $\Z$-graded, see Proposition~\ref{regata}. As an application, we show that if two OTW-subshifts are conjugate then the graphs associated with the corresponding subshifts satisfy Condition~(Y), see Theorem~\ref{conjugacy}. We employ the latter criterion for two specific examples and conclude that neither the associated OTW-subshifts are conjugate, nor there is an isometric conjugacy between the associated subshifts built using the $\frac{1}{2^i}$-metric (which induces the product topology), see Corollary~\ref{saidarapida} for details.

We finish the paper in Section~6,  with a brief examination of the socle of the subshift C*-algebra. Specifically, for a given subshift, we show that the socle of the subshift algebra is contained in the socle of the subshift C*-algebra and we establish that this inclusion may be proper. This is the same behavior presented by Leavitt path algebras, see \cite[Theorem~3.6]{Socle1}.  However, the example we provide to illustrate the possibility of strict containment differs from the example in \cite{Socle1}, which relies on a graph with sinks (a setting not permitted for subshift algebras).

\section{Preliminaries}\label{eleicoes}

 In this section, we recall the definition of the unital subshift algebra and some relevant results regarding it. We start setting up notation and recalling the notion of a subshift.
 
Throughout the paper, $R$ stands for a commutative unital ring, $\nn=\{0,1,2,\ldots\}$, and $\nn^*=\{1,2,\ldots\}$.

\subsection{Symbolic Dynamics}\label{symbolic}

Let $\alf$ be a non-empty set, referred to as an \emph{alphabet}. The shift map $\sigma$ is a map from $\alf^\N$ to itself, defined as $\sigma(x)=y$, where $x=(x_n){n\in \N}$ and $y=(x{n+1})_{n\in \N}$. We designate elements of $\alf^*=\bigcup_{k=0}^\infty \alf^k$ as \emph{blocks} or \emph{words}, with $\omega$ representing the empty word. Additionally, we define $\alf^+=\alf^*\setminus\{\eword \}$. 
For any $\alpha\in\alf^*\cup\alf^{\N}$, $|\alpha|$ denotes its length. Given $1\leq i,j\leq |\alpha|$, we define $\alpha_{i,j}:=\alpha_i\cdots\alpha_j$ if $i\leq j$, and $\alpha_{i,j}=\eword$ if $i>j$. 
If $\beta\in\alf^*$, then $\beta\alpha$ denotes concatenation. For a block $\alpha\in \alf^k$, $\alpha^\infty$ represents the infinite word $\alpha \alpha \ldots$. A subset $\osf \subseteq \alf^\N$ is \emph{invariant} for $\sigma$ if $\sigma (\osf)\subseteq \osf$. For an invariant subset $\osf \subseteq \alf^\N$, we define $\CL_n(\osf)$ as the set of all words of length $n$ that appear in some sequence of $\osf$, that is, $$\CL_n(\osf):=\{(a_0\ldots a_{n-1})\in \alf^n:\ \exists \ x\in \osf \text{ s.t. } (x_0\ldots x_{n-1})=(a_0\ldots a_{n-1})\}.$$ It is evident that $\CL_n(\alf^\N)=\alf^n$, and it is always the case that $\CL_0(\osf)={\omega}$.
The \emph{language} of $\osf$, denoted as $\lang$, encompasses all finite words that manifest within some sequence of $\osf$. Formally,
$$\lang:=\bigcup_{n=0}^\infty\CL_n(\osf).$$
For elements $c,d\in \lang$, we use the notation $c_{|c|}\neq d_{|d|}$ meaning that the last letter of $c$ is different from the last letter of $d$.

Given $F\subseteq \alf^*$, the \emph{subshift} $\osf_F\subseteq \alf^\N$ is the set of all sequences $x$ in $\alf^\N$ such that no word of $x$ belongs to $F$. When the context is clear, we denote $\osf_F$ by $\osf$. 
The key sets that are used in the definition of the algebra associated with a subshift are defined below. 

\begin{definition}\label{Arapaima}
Let  $\osf$ be a subshift for an alphabet $\alf$. Given $\alpha,\beta\in \lang$, define \[C(\alpha,\beta):=\{\beta x\in\osf:\alpha x\in\osf\}.\] In particular, the set  $C(\eword,\beta)$ is denoted by $Z_{\beta}$ and called a \emph{cylinder set}, and the set  $C(\alpha,\eword)$ is denoted by $F_{\alpha}$ and called a \emph{follower set}. Notice that $\osf=C(\eword,\eword)$. 
\end{definition}

\subsection{Unital algebras of subshifts}\label{unital}

In this section, we recall the definition of the unital algebra associated with a general subshift $\osf$, as done in \cite{BGGV}. We start with the definition of the Boolean algebra associated with the sets of the form $C(\alpha,\beta)$.

\begin{definition}\label{diachuvoso}
Let $\osf$ be a subshift. Define $\TCB$ to be the Boolean algebra of subsets of $\osf$ generated by all $C(\alpha,\beta)$ for $\alpha,\beta\in\lang$, that is, $\TCB$ is the collections of sets obtained from finite unions, finite intersections, and complements of the sets $C(\alpha,\beta)$.
\end{definition}

We can now recall the definition of the unital algebra associated with a subshift.

\begin{definition}\label{universal properties}
Let $\osf$ be a subshift. We define the \emph{unital subshift algebra} $\ualgshift$ as the universal unital $R$-algebra  with generators $\{p_A: A\in\TCB\}$ and $\{s_a,s_a^*: a\in\alf\}$, subject to the relations:
\begin{enumerate}[(i)]
    \item $p_{\osf}=1$, $p_{A\cap B}=p_Ap_B$, $p_{A\cup B}=p_A+p_B-p_{A\cap B}$ and $p_{\emptyset}=0$, for every $A,B\in\TCB$;
    \item $s_as_a^*s_a=s_a$ and $s_a^*s_as_a^*=s_a^*$ for all $a\in\alf$;
    \item $s_{\beta}s^*_{\alpha}s_{\alpha}s^*_{\beta}=p_{C(\alpha,\beta)}$ for all $\alpha,\beta\in\lang$, where $s_{\eword}:=1$ and, for $\alpha=\alpha_1\ldots\alpha_n\in\lang$, $s_\alpha:=s_{\alpha_1}\cdots s_{\alpha_n}$ and $s_\alpha^*:=s_{\alpha_n}^*\cdots s_{\alpha_1}^*$.
\end{enumerate}
\end{definition}

It follows from the third relation of the previous definition that $s_\alpha s_\alpha^*=Z_\alpha$, and $s_\alpha^*s_\alpha=F_\alpha$, for each $\alpha \in \lang$. Recall that the relative range of $(A,\alpha)$, where $\alpha\in \lang$ and $A\in \TCB$, is given by \[r(A,\alpha)=\{x\in\osf: \alpha x\in A\}.\] 

Next, we recall some results that will be necessary in our work.

\begin{lemma}[\cite{reductiontheoremofsubshifts}, Lemma 3.7]\label{chucrute} Let $\osf$ be a subshift, $a,b\in \alf$, and $\gamma,\alpha \in \lang$.
\begin{enumerate}[\hspace{.3cm} \rm (1)]
\item If $\beta:=b\gamma\in \lang$, then $s_\beta^*s_a=\delta_{b,a}s_\gamma^*p_{F_a}$.\vspace{.2cm}
\item If $A\in \TCB$, then $p_As_\alpha=s_\alpha p_{r(A,\alpha)}$ and $s_\alpha^*p_A=p_{r(A,\alpha)}s_\alpha^*$.
\end{enumerate}
\end{lemma}

By \cite{BGGV}, $\ualgshift$ has a $\Z$-grading given by 
    \[\ualgshift_n = \vecspan_R\{s_\alpha p_A s_\beta^* : \alpha,\beta \in \lang,\ A\in\TCB \ \mbox{and} \ |\alpha|-|\beta|=n\}, n \in \Z.\]
It is also graded by the free group on the alphabet. This grading arises from the isomorphism of the algebra with a certain partial skew group ring, as stated below.

\begin{theorem}[\cite{BGGV},Theorem~5.9]\label{thm:set-theoretic-partial-action}
Let $\osf$ be a subshift. Then, $\ualgshift\cong\udalgshift\rtimes_{\tau}\F$ via an isomorphism $\Phi$ that sends $s_a$ to $1_a\delta_a$ and $s^*_a$ to $1_{a^{-1}}\delta_{a^{-1}}$.
\end{theorem}

Another important result regarding $\ualgshift$ is the so-called Reduction Theorem (which was proved in \cite{reductiontheoremofsubshifts}). We recall this result below, after providing the necessary definitions.

\begin{definition}[\cite{reductiontheoremofsubshifts}, Definitions 4.1 and 4.2]\label{cyclexit} Let $\osf$ be a subshift, $\alpha\in \lang\setminus \{w\}$, and $\emptyset \neq A\in \TCB$. The pair $(A,\alpha)$ is called  a \emph{cycle} if $A\subseteq r(A,\alpha)$.
We say that a cycle $(A,\alpha)$ has an \emph{exit} if $A\neq \{\alpha^\infty\}$. Otherwise, in case $A=\{\alpha^\infty\}$, we say that $(A,\alpha)$ is a \emph{cycle without exit}.
\end{definition}

\begin{definition}[\cite{reductiontheoremofsubshifts}, Definition 4.4]\label{mini}
Let $\osf$ be a subshift. We say that a cycle without exit $(A,\alpha)$ is \emph{minimal} if there is no element $\beta\in \lang$, with $1\leq |\beta|<|\alpha|$, such that $\alpha=\beta^k$ for some $k\geq 2$.
\end{definition}

\begin{theorem}[\cite{reductiontheoremofsubshifts}, Reduction Theorem]\label{reduction theorem}  Let $\osf$ be a subshift and $x\in \ualgshift$ be a non-zero element. Then, there exists $\mu,\nu\in \ualgshift$ such that  $\mu x\nu\neq 0$ and either:
\begin{enumerate}[\hspace{.3cm} \rm (1)]
    \item $\mu x\nu=\gamma p_D$ with $D\in \TCB$, or
    \item $\mu x\nu=\gamma_1 p_A +\sum\limits_{i=2}^k \gamma_i s_{\beta^{q_i}}p_A$, where $(A,\beta)$ is a minimal cycle without exit, $q_i\in \N\setminus \{0\}$, and $0\neq \gamma_i\in R$.
\end{enumerate}

\end{theorem}





\section{Minimal left ideals in $\ualgshift$}

In this section, we characterize, up to isomorphism, the minimal left ideals of $\ualgshift$ (see Corollary~\ref{laVuelta}). To begin, we show a couple of results on the structure of left ideals of $\ualgshift$. As in the previous section, unless otherwise stated, $R$ stands for a commutative unital ring.

\begin{proposition}\label{leftideal} Let $X$ be a subshift. Then,
\begin{enumerate}
\item For each $x\in \ualgshift$, $\ualgshift x$ is a left ideal of $\ualgshift$.
\item If $0\neq I\subseteq \ualgshift$ is a left minimal ideal then there exists some $x\in \ualgshift$ such that $I=\ualgshift x$.
\item For all $A,B\in \TCB$ with $\emptyset\neq  B\subsetneq A$, it holds that $0\neq \ualgshift p_B\subsetneq \ualgshift p_A$.
\end{enumerate}
\end{proposition}

\begin{proof} The first item is clear, so we begin with the proof of the second one. 

Let $0\neq I\subseteq \ualgshift$ be a left minimal ideal, and let $0\neq x\in I$. Then, $\ualgshift x$ is a left ideal of $\ualgshift$ which is contained in $I$. Since $\ualgshift x\neq 0$ and $I$ is minimal, we obtain that $I=\ualgshift x$.

Next, we prove the third item. Notice that $0\neq \ualgshift p_B$, since $p_B\in \ualgshift p_B$. Furthermore, for any $z \in \ualgshift$ we have $zp_B = zp_Bp_A$ (since $B \subseteq A$), which implies that $\ualgshift p_B \subseteq \ualgshift p_A$. It remains to prove that the inclusion is proper.
Suppose that $\ualgshift p_B= \ualgshift p_A$. This implies the existence of some $z \in \ualgshift$ such that $zp_B = p_A$. Define $C=B^c\cap A$, which is an element of $\TCB$. Observe that $p_Bp_C = 0$ since $B \cap C = \emptyset$, and $p_Ap_C = p_C \neq 0$ because $C \neq \emptyset$. Therefore, $0=zp_Bp_C=p_Ap_C\neq 0$, a contradiction.
We conclude that $\ualgshift p_B \neq \ualgshift p_A$.

\end{proof}

\begin{proposition}\label{prop1} Let $X$ be a subshift. Then, for each $A\in \TCB$ with $|A|>1$, $\ualgshift p_A$ is a left ideal of $\ualgshift$ which is not minimal.
\end{proposition}

\begin{proof} Using the first item of Proposition~\ref{leftideal}, we obtain that $\ualgshift p_A$ is a left ideal of $\ualgshift$. Given that $|A|>1$, there exists $x,y\in A$ with $x\neq y$. Write $x=x_1x_2...$ and $y=y_1y_2...$, where $x_i,y_i\in \alf$ for each $i$. Since $x\neq y$, there exists an index $j\in\nn$ such that $x_j\neq y_j$. Let $B=Z_{x_{(1,j)}}\cap A$. Since $x\in B$ and $y\notin B$, we have that $\emptyset \neq B\subsetneq A$. Consequently, using the third
item of Proposition~\ref{leftideal}, we obtain that $\ualgshift p_A$ is not minimal.
\end{proof}

Given the above, we are interested in the left ideals of the form $\ualgshift p_A$, where $|A|=1$. Before proceeding, we review the definition of a line path below. 

\begin{definition}[\cite{Irreduciblerepresentations}, Definition 4.1]\label{linepath} Let $\osf$ be a subshift. We say that an element $p=a_0a_1a_2a_3...\in \osf$ is a line path if $Z_{a_{0}}=\{p\}$ and, for every $\beta\in \lang$ and $k\in \N$, we have that $\beta^\infty\neq a_ka_{k+1}a_{k+2}\ldots$. We denote  by $\mathcal{P_\osf}$ the set of all the elements $A\in \TCB$ such that $A=\{p\}$, for some line path $p\in \osf$.
\end{definition}

Line paths are associated with minimal left ideals. In fact, it is proved in \cite[Corollary~4.8]{Irreduciblerepresentations} that, for each $A\in \mathcal{P_\osf}$, $\ualgshift p_A$ is a minimal left ideal. However, there are other minimal ideals of the form $\ualgshift p_B$, with $B\notin \mathcal{P_\osf}$. To show this, we need the following definition.






\begin{definition}\label{chuvacalor}
Let $\osf$ be a subshift. Denote by $\mathcal{Q_\osf}$ the set of all the single point sets $\{p\}\in \TCB$  such that $p\neq \alpha \beta^\infty$ for each $\alpha, \beta\in \lang$.
We call an element of $\mathcal{Q_\osf}$ an irrational path.
\end{definition}
  
Notice that $\mathcal{P_\osf}\subseteq \mathcal{Q_\osf}$ but, in general, $\mathcal{P_\osf}\neq \mathcal{Q_\osf}$ as we show in the next example.

\begin{example}\label{irrational paths are greater than line paths}
    Let $\alf=\{a,b,c\}$ be an alphabet with 3 letters, and let $x\in \alf^\infty$ be the element
    $$x=bcb^2cb^3cb^4c....$$
    
 Define $$\osf=\{a^\infty,  b^\infty, ax, cx\}\cup \{\sigma^n(x):n\in \N\},$$ which is a subshift.
 
This subshift has no line paths since the cardinalities of $Z_a, Z_b$ and $Z_c$ are all greater than one. However, 
 $\mathcal{Q_\osf}\neq \emptyset$ since $\TCB \ni C(c,a)=\{ax\}$ is an element of $\mathcal{Q_\osf}$.
    
\end{example}
The lemma presented below allows us to establish a dichotomy for the set of irrational paths in a subshift: either $\mathcal{Q_\osf}$ is empty or it has infinitely many elements. 

\begin{lemma}\label{many elements in Q} Let $\osf$ be a subshift. If $A=\{p\}\in \mathcal{Q_\osf}$ then $\{\sigma^{n}(p)\}\in \mathcal{Q_\osf}$ for each $n\in \N$ and, moreover, $\sigma^n(p)\neq \sigma^m(p)$ for each $m\neq n$. Consequently, if $\mathcal{Q_\osf}$ is nonempty then it contains infinitely many elements.     
\end{lemma}

\begin{proof} Let $A=\{p\}\in \mathcal{Q_\osf}$, where $p=p_0p_1p_2...\in \osf$, $n\in \N$, and define $a=p_0p_1..p_{n-1}$. Notice that $\{\sigma^n(p)\}=r(A,a)\in \TCB$. Suppose that $\{\sigma^n(p)\}\notin \mathcal{Q_\osf}$, that is, $\sigma^n(p)=\alpha \beta ^\infty$ for some $\alpha, \beta \in \lang$. Then, $p=a \alpha \beta^\infty$ and hence $\{p\}\notin \mathcal{Q_\osf}$, a contradiction. 
Therefore, $\{\sigma^n(p)\}\in \mathcal{Q_\osf}$. 

Now, suppose that $\sigma^n(p)=\sigma^m(p)$ for some $m\neq n$, and suppose, without loss of generality, that $n<m$. Let $z=\sigma^n(p)=\sigma^m(p)$ and let $b=p_np_{n+1}...p_{m-1}$. Then,  $z=\sigma^n(p)=b\sigma^m(p)=bz$ and consequently $z=b^\infty$. Hence, $p=p_0p_1...p_{n-1}z=p_0p_1...p_{n-1}b^\infty$, which implies that $\{p\}\notin \mathcal{Q_\osf}$, a contradiction. Therefore, $\sigma^n(p)\neq \sigma^m(p)$. 
\end{proof}


Recall that a nonzero idempotent $e$ in an algebra $A$ is minimal if  $eAe$ is a division ring, see \cite[Definition~30.1]{SocleBook}. Furthermore, if $A$ is semi-prime then $L$ is a minimal left ideal of $A$ if and only if $L = A e$, where $e$ is a minimal idempotent in $A$, see  \cite[Proposition~30.6]{SocleBook}. We will use these concepts below, to prove that left ideals associated with irrational paths are minimal. From now on, we assume that the base ring $R$ is a field, so that $\ualgshift$ is semiprime (in fact, it is enough to ask that $R$ has no zero divisors to obtain semiprimeness, see [\cite[Corollary~5.6]{reductiontheoremofsubshifts}).

\begin{proposition}
\label{elements of Q generate minimal ideals} Let $\osf$ be a subshift and $R$ be a field. Then, for each $A\in \mathcal{Q_\osf}$, $\ualgshift p_A$ is a left minimal ideal of $\ualgshift$.    
\end{proposition}

\begin{proof} Fix an $A=\{p\}\in \mathcal{Q_\osf}$, and let $\alpha, \beta \in \lang$ and $B\in \TCB$ be such that $p_As_\alpha p_B s_\beta^* p_A\neq 0$. Then $\alpha$ is an initial path of $p$, since otherwise $p_A s_\alpha=0$. Similarly, $\beta$ is an initial path of $p$. Write $p=\alpha x$ and $p=\beta y$, where $x,y \in \osf$. 

Then, since
$$0\neq p_As_\alpha p_B s_\beta^* p_A=s_\alpha p_{r(A,\alpha)}p_B p_{r(A,\beta)} s_\beta^*,$$ we have that $\emptyset\neq r(A,\alpha)\cap B \cap r(A,\beta)$ and hence $r(A,\alpha)=r(A,\alpha)\cap B\cap  r(A,\beta)=r(A,\beta)$. Notice that $r(A,\alpha)=\{x\}=\{\sigma^{|\alpha|}(p)\}$ and $r(A,\beta)=\{y
\}=\{\sigma^{|\beta|}(p)\}$, so that $\sigma^{|\alpha|}(p)=\sigma^{|\beta|}(p)$. From Lemma~\ref{many elements in Q}, we get that $|\alpha|=|\beta|$ and, consequently, $\alpha=\beta$. Hence, 
$$p_As_\alpha p_B s_\beta^* p_A=s_\alpha p_{r(A,\alpha)} p_B p_{r(A,\beta)} s_\beta^* = s_\alpha p_{r(A,\alpha)}p_{r(A,\alpha)} s_\alpha^*= $$
$$=p_A s_\alpha s_\alpha^* p_A= 
p_A p_{Z_\alpha} p_A=p_A.$$

Therefore, as elements of the form $s_\alpha p_B s_\beta^*$ generate $\ualgshift$ (see the comment about the $\Z$-grading of $\ualgshift$ below Lemma~\ref{chucrute}), we obtain that $p_A \ualgshift p_A=R p_A$. This means that $p_A$ is a minimal idempotent, and so, by \cite[Proposition~30.6]{SocleBook}, we conclude that $\ualgshift p_A$ is a minimal left ideal.
\end{proof}

Next, we completely characterize minimal left ideals associated with elements of $\mathcal U$.

\begin{proposition}\label{prop2}
    Let $\osf$ be a subshift, $R$ be a field, and $A\in \TCB$. Then, $\ualgshift p_A$ is a left minimal ideal if and only if $A\in \mathcal{Q_\osf}$.
\end{proposition}

\begin{proof} 

Let $A\in \mathcal{Q_\osf}$. Then, by Proposition~\ref{elements of Q generate minimal ideals}, $\ualgshift p_A$ is minimal. 

To prove the converse, let $A\in \TCB$ and assume that $\ualgshift p_A$ is minimal. If $|A|>1$ then, by Proposition~\ref{prop1}, $\ualgshift p_A$ is not minimal, a contradiction. Hence, $|A|=1$ and we write $A=\{p\}$.
Suppose that $A\notin \mathcal{Q_\osf}$. Then, $p=\alpha\mu^\infty$, with $\mu, \alpha \in \lang$. 

Suppose first that $|\alpha|=0$, that is, $p=\mu^\infty$. Let $\beta\in \lang$ be the element of minimal length such that $\mu=\beta^k$ for some $k\in \N$. Then, $\mu^\infty=p=\beta^\infty$, and the pair $(A,\beta)$ is a minimal cycle without exit (see Definition~\ref{mini}). Applying \cite[Lemma~5.3]{reductiontheoremofsubshifts}, we obtain that $p_A \ualgshift p_A$ and $R[x,x^{-1}]$ are isomorphic algebras, where $R[x,x^{-1}]$ denotes the Laurent polynomials ring. Since $R[x,x^{-1}]$ contains nontrivial (two-sided) ideals, there exists a (two-sided) ideal $0\neq I\subseteq p_A \ualgshift p_A$ with $I\neq p_A \ualgshift p_A$. Let $0\neq x\in I$ and $J=\ualgshift x$, which is a left ideal of $\ualgshift$. Notice that $0\neq J\subseteq \ualgshift p_A$, since $x=p_Axp_A$. Aiming to show that $J\neq \ualgshift p_A$, we suppose that $J=\ualgshift p_A$. Then, there exists $y\in \ualgshift$ such that $yx=p_A$. Since  $x=p_Axp_A$, we have that $yp_Axp_A=p_A$. Multiplying this equality on the left by $p_A$, we get that $(p_Ayp_A)(p_Axp_A)=p_A$, and so $p_Ayp_Ax=p_A$. Since $p_Ayp_A\in p_A\ualgshift p_A$, $x\in I$, and $I$ is a two-sided ideal of $p_A \ualgshift p_A$, we conclude that $p_A=p_Ayp_Ax\in I$. Hence, $p_A \ualgshift p_A\subseteq I$, which is impossible since $I\subsetneq p_A \ualgshift p_A$. So, we get that $J\neq \ualgshift p_A$ and, consequently, $\ualgshift p_A$ is not minimal.

Next, suppose that $p=\alpha\mu^\infty$, with $|\alpha|>0$. Define $B=\{\mu^\infty\}$ and notice that $r(A,\alpha)=\{\mu^\infty\}$. Then, $p_A s_\alpha=s_\alpha p_{r(A,\alpha)}=s_\alpha p_B$ and, similarly, $s_\alpha^* p_A=p_B s_\alpha^*$. From the previous paragraph we get that $\ualgshift p_B$ is not minimal. Let $J$ be a left ideal of $\ualgshift$ such that 
$0\neq J\subsetneq \ualgshift p_B$.
Define $I=Js_\alpha^*$, which is a left ideal of $\ualgshift$. Notice that $I=Js_\alpha^* =Jp_B s_\alpha^* =Js_\alpha^* p_A$, so that $I\subseteq \ualgshift p_A$. Let $0\neq x\in J$. Then, $x=xp_B=xp_B p_{F_\alpha}=xp_B s_\alpha^*s_\alpha$ and therefore $0\neq xp_B s_\alpha^*$. So, $I\neq 0$. Aiming to show that $I\neq \ualgshift p_A$, suppose that $I=\ualgshift p_A$. Then, $p_A\in I=Js_\alpha^*$ and there is $y\in J$ such that  $p_A=y p_B s_\alpha^*$. Multiplying this equality on the right by $s_\alpha$ we get that $p_A s_\alpha=yp_Bs_\alpha^* s_\alpha=yp_B$ and, since $p_A s_\alpha=s_\alpha p_B$, we conclude that $s_\alpha p_B=y p_B$. Now, multiplying this equality on the left by $s_\alpha^*$, we get that $s_\alpha^* s_\alpha p_B=s_\alpha^*y p_B$, which is an element of $J$. As $s_\alpha^* s_\alpha p_B=p_B$, we obtain that $p_B\in J$ and, consequently, $J=\ualgshift p_B$, which is a contradiction. Therefore, $I\neq \ualgshift p_A$  and $\ualgshift p_A$ is not minimal.
\end{proof}

Our next goal is to completely characterize the left minimal ideals of $\ualgshift$. For this, we need the following lemma.

\begin{lemma}\label{minimallemma}
Let $0\neq x\in \ualgshift$ be such that $\ualgshift x$ is a left minimal ideal of $\ualgshift$. Then, for each $\mu',\nu'\in \ualgshift
$ with $\mu'x\nu'\neq 0$, $\ualgshift x$ and $\ualgshift \mu' x \nu'$ are isomorphic as left $\ualgshift$-modules (and consequently $\ualgshift \mu' x \nu'$ is also minimal).
\end{lemma}

 \begin{proof}

 First, observe that $\ualgshift \mu' x=\ualgshift x$. Indeed, this follows from the minimality of $\ualgshift x$ and the fact that $0\neq \ualgshift \mu' x$ is a left $\ualgshift$ ideal contained in $\ualgshift x$.


Define $\varphi:\ualgshift \mu' x\rightarrow \ualgshift \mu' x \nu'$ by $\varphi(a)=a \nu'$, for each $a\in \ualgshift \mu' x$. Notice that this map is a surjective left $\ualgshift$-module homomorphism.  
We show that $\varphi$ is injective. Since $\varphi(\mu' x)=\mu' x \nu'\neq 0$, we have that $ker(\varphi)\neq \ualgshift \mu' x$. Moreover, $Ker(\varphi)$ is a left $\ualgshift$ ideal contained in $\ualgshift \mu' x =\ualgshift x$. By the minimality of  $\ualgshift x$, we conclude that  $ker(\varphi)=0$. Therefore, $\varphi$ is an $\ualgshift$-left module isomorphism.
\end{proof}

\begin{proposition}\label{minimal isomorphism} Let $\osf$ be a subshift, $R$ be a field, and let $x\in \ualgshift$ be such that $\ualgshift x$ is a minimal left ideal. Then, $\ualgshift x$ is isomorphic (as a left $\ualgshift$-module) to $\ualgshift p_D$ for some $D\in \mathcal{Q_\osf}$.
\end{proposition}

\begin{proof} Applying the Reduction Theorem (Theorem~\ref{reduction theorem}) for $x$, we obtain $\mu,\nu \in \ualgshift$ such that 
$\mu x \nu\neq 0$ and $\mu x \nu=\lambda p_D$, where $D\in \TCB$ 
and $\lambda \in R$, or $\mu x\nu=\gamma_1 p_A +\sum\limits_{i=2}^k \gamma_i s_{\beta^{q_i}}p_A$, 
where $(A,\beta)$ is a minimal cycle without exit (so $A=\{\beta^\infty\}$), $q_i\in \N\setminus \{0\}$, and $0\neq \gamma_i\in R$.

By Lemma \ref{minimallemma}, $\ualgshift x$ and $\ualgshift \mu x \nu$ are isomorphic as left $\ualgshift$-modules. So, if  $\mu x \nu=\lambda p_D$ then $\ualgshift x$ is isomorphic to $\ualgshift \lambda p_D$. As $R$ is a field, $\ualgshift \lambda p_D = \ualgshift p_D$ and, by Proposition~\ref{prop2}, we get that $D\in \mathcal{Q_\osf}$.

Next, we show that the second possibility for $\mu x \nu$ does not happen. For this, suppose that $\mu x\nu=\gamma_1 p_A +\sum\limits_{i=2}^k \gamma_i s_{\beta^{q_i}}p_A$ and, to simplify notation, let $z=\mu x \nu$. 

From [\cite{reductiontheoremofsubshifts}, Lemma 5.3] there exists an $R$-isomorphism $\psi:p_A \ualgshift p_A \rightarrow R[x,x^{-1}]$. Let 
$I=p_A \ualgshift z$, which is a nonzero two-sided ideal of $p_A \ualgshift p_A$ (to see that it is an ideal on the right, use that  $z=p_A z p_A$ since $(A,\beta)$ is a minimal cycle without exit, and that $R[x,x^{-1}]$ is commutative).

Next, we show that $I$ is a minimal ideal of $p_A \ualgshift p_A$. Let $0\neq J$ be a $p_A \ualgshift p_A$ ideal such that $J\subseteq I$. Then, $0\neq \ualgshift J$ is an $\ualgshift$ left ideal contained in $\ualgshift z$. Since $\ualgshift z$ is minimal (by Lemma \ref{minimallemma}) we obtain that  $\ualgshift J=\ualgshift z$. Consequently, $p_A \ualgshift J=p_A\ualgshift z = I$ and, since $p_A \ualgshift  J= p_A \ualgshift p_A J=J$, we conclude that $J=I$. Therefore, $I$ is minimal. 

The minimality of $I$ implies that $\varphi(I)$ is a minimal ideal of $R[x,x^{-1}]$, which is a contradiction, since $R[x,x^{-1}]$ has no minimal ideals. Therefore, the case $\mu x\nu=\gamma_1 p_A +\sum\limits_{i=2}^k \gamma_i s_{\beta^{q_i}}p_A$, 
is not possible, as desired.

\end{proof}

\begin{corollary}\label{laVuelta} Let $\osf$ be a subshift, $R$ be a field and $I$ be a left ideal in $\ualgshift$. Then, $I$ is minimal if and only if it is isomorphic (as a left $\ualgshift$-submodule) to $\ualgshift p_A$, for some $A\in \mathcal{Q_\osf}$.
\end{corollary}

\begin{proof} If $I$ is isomorphic as a left $\ualgshift$-ideal to $\ualgshift p_A$, for some $A\in \mathcal{Q_\osf}$, then it is minimal by Proposition~\ref{prop2}. For the other statement, suppose that $I$ is a left $\ualgshift$ minimal ideal. By Proposition~\ref{leftideal}, there exists $x\in \ualgshift$ such that $I=\ualgshift x$. The result now follows from Proposition~\ref{minimal isomorphism}.
\end{proof}

\section{The Socle of $\ualgshift$}

Recall that the (left) socle of an algebra $B$, denoted by $Soc(B)$, is the sum of all the left minimal ideals in $B$ (hence $Soc(B)$ is also a left ideal). If $B$ is semiprime, then $Soc(B)$ coincides with the sum of all the right minimal ideals in $B$, as proved in \cite[Chapter~3, Proposition~4]{Lambek}). Hence, for semiprime algebras the socle is a two-sided ideal. Recall that we are assuming that $R$ is a field and so $\ualgshift$ is semiprime.

Our goal in this section is to characterize the socle of $\ualgshift$ as the sum of the minimal left ideals associated with elements of $\mathcal{Q_\osf}$. By Proposition~\ref{prop2}, we have that 
$$\sum\limits_{D\in \mathcal{Q_\osf}}\ualgshift p_D\subseteq Soc(\ualgshift).$$

To prove the reverse inclusion, we first need a lemma.

\begin{lemma}\label{parouAchuva!} Let $\osf$ be a subshift, $x\in \osf$, and $D=\{x\}\in \TCB$.
\begin{enumerate}
\item If $\beta \in \lang$ is such that $D\subseteq F_\beta$, then $\{\beta x \}\in \TCB$.
\item For each $\alpha, \beta \in \lang$ and $A\in \TCB$ such that $p_D s_\alpha p_A s_\beta^*\neq 0$, there exists $y\in \osf$ such that $x=\alpha y$, $\{y\}$ and $\{\beta y \}\in \TCB$, and 
$$p_D s_\alpha p_A s_\beta^*=s_\alpha s_\beta^* p_{\{\beta y\}}.$$ 

\item For each $\alpha, \beta \in \lang$ and $A\in \TCB$, if $s_\alpha p_A s_\beta^*p_D\neq 0$ then 
$s_\alpha p_A s_\beta^*p_D=s_\alpha s_\beta^* p_D.$

\end{enumerate}
\end{lemma}

\begin{proof} We begin with the first item. Let $\Phi$ be the isomorphism of Theorem~\ref{thm:set-theoretic-partial-action}. Notice that 
$$\Phi(s_\beta p_D s_\beta^*)=1_{\beta} \delta_{\beta} 1_D 1_{\beta^{-1}} \delta_{\beta^{-1}}=1_{\{\beta x\}}\delta_0.$$ Recall from \cite[Section~5]{BGGV} that $\udalgshift$ is generated by the characteristic functions of the sets $C(\alpha, \beta)$, with $\alpha, \beta \in \lang$. Consequently, $1_{\{\beta x\}}$ is a finite sum of the form $1_{\beta x}=\sum\limits_{i\in F} \lambda_i 1_{A_i}$, where $A_i \in \TCB$ and $F$ is a finite set. Moreover, since $\TCB$ is a Boolean algebra, we can suppose that $A_i\cap A_j=\emptyset$ for each $i\neq j$ (see \cite[Lemma~3.5]{BCGWLabel}). Since $\{\beta x\}$ is a set with only one element, and all the $A_i$ are disjoint, we obtain that $F$ is also a set with only one element, say $F=\{i\}$. Hence, $1_{\beta x}=\lambda_i 1_{A_i}$, which implies that $\lambda_i=1$ and $\{\beta x\}=A_i \in \TCB$.

Next, we prove the second item. Let $\alpha, \beta \in \lang$ and $A\in \TCB$. Notice that if $x\notin Z_\alpha$ then $p_Ds_\alpha=0$ and hence $p_D s_\alpha p_A s_\beta^*=0$. Therefore, if $p_D s_\alpha p_A s_\beta^*\neq 0$ then $x\in Z_\alpha$, say $x=\alpha y$ for some $y\in \osf$. 

Thus, $$p_D s_\alpha p_A s_\beta^*=p_{\{\alpha y\}}s_\alpha p_A s_\beta^*=s_\alpha p_{\{y\}}p_A s_\beta^*.$$ 
From the last term in the equality above, we obtain that $p_{\{y\}} p_A=0$ if $y\notin A$. Therefore, we may suppose that $y\in A$. Hence, 
$$s_\alpha p_{\{y\}}p_A s_\beta^*=s_\alpha p_{\{y\}}s_\beta^*.$$ 
In case that $y\notin F_\beta$, the above implies that $p_{\{y\}}s_\beta^*=0$, and so we may suppose that $y \in F_\beta$. Observe that $\{y\} \in \TCB$ since $\{y\}=r(\{\alpha y\},\alpha)$ and so, from the first item of this lemma, we get that $\{\beta y\} \in \TCB$. Thus, from the second item of Lemma~\ref{chucrute}, we obtain that $p_{\{y\}}s_\beta^*=s_\beta^* p_{\{\beta y\}}$ and hence $s_\alpha p_{\{y\}} s_\beta^*=s_\alpha s_\beta^* p_{\{\beta y\}}$. Thus, $p_Ds_\alpha p_A s_\beta^*=s_\alpha s_\beta^*p_{\{\beta y\}}$ as desired. 

To prove the third item, let $\alpha, \beta \in \lang$ and $A\in \TCB$ be such that $s_\alpha p_A s_\beta^*p_D\neq 0$. Since $0\neq s_\alpha p_A s_\beta^* p_D=s_\alpha p_A p_{r(D,\beta)} s_\beta^*$  we have that $r(D,\beta)\subseteq A$. Hence, $s_\alpha p_A p_{r(D,\beta)} s_\beta^*=s_\alpha p_{r(D,\beta)} s_\beta^*=s_\alpha s_\beta^* p_D$, where the last equality follows from the second item of Lemma~\ref{chucrute}.
    
\end{proof}

Next, we prove that the socle of a subshift algebra coincides with the sum of the minimal left ideals associated with elements of $\mathcal{Q_\osf}$.

\begin{theorem}\label{socle description} Let $\osf$ be a subshift and $R$ be a field. Then, $$Soc(\ualgshift)=\sum\limits_{A \in \mathcal{Q_\osf}}\ualgshift p_A.$$ 
\end{theorem}

\begin{proof} Let $0\neq I$ be a left minimal ideal of $\ualgshift$. By Lemma~\ref{leftideal}, there exists $0\neq x\in \ualgshift$ such that $I=\ualgshift x$. For this $x$, let $\mu,\nu$ be as in the Reduction Theorem (Theorem~\ref{reduction theorem}) and use Lemma~\ref{minimallemma} to conclude that $\ualgshift \mu x \nu$ is minimal. Then, by the proof of Proposition~\ref{minimal isomorphism}, we get that $\mu x \nu=\lambda p_D$, for some $0\neq \lambda \in R$ and $D\in \TCB$. Notice that since $\ualgshift x$ is minimal, we have that $\ualgshift \mu x =\ualgshift x$. Define $x'=\mu x$  and let $\varphi :\ualgshift x' \rightarrow \ualgshift p_D$ be the ($\ualgshift$ left) isomorphism of the proof of Lemma \ref{minimallemma}, that is, $\varphi(z x')=z x' \nu=z \lambda p_D$. Observe that $\varphi(x')=\lambda p_D$ and, since $\varphi$ is a left $\ualgshift$-isomorphism, this implies that
$$x'=\varphi^{-1}(\varphi(x'))=\varphi^{-1}(\lambda p_D)=\varphi^{-1}( p_D\lambda p_D)=p_D \varphi^{-1}(\lambda p_D)=p_D x'.$$

Write $x'$ as a finite sum of the form $x'=\sum \limits \lambda_i s_{\alpha_i} p_{A_i} s_{\beta_i}^*$, where $A_i \in \TCB$, $\alpha_i, \beta_i \in \lang$, and $\lambda_i\in R$ for each $i$. From the second item of Lemma~\ref{parouAchuva!}, we obtain that $p_D x'$ has the form 
$$p_D x'=\sum \lambda_i s_{\alpha_i} s_{\beta_i}^* p_{\{\beta_i y_i\}},$$ where $D=\{\alpha_i y_i\}$ for each $i$ (notice that we can apply Lemma~\ref{parouAchuva!} because, as $\ualgshift p_D$ is minimal, $D$ is a one point set). Since $\{\alpha_i y_i\} \in \mathcal{Q_\osf}$ for each $i$, we have that $\{\beta_i y_i\} \in \mathcal{Q_\osf}$ for each $i$ and hence $$x'=p_D x'=\sum \limits s_{\alpha_i}s_{\beta_i}^*p_{\{\beta_i y_i\}}\in \sum \limits_{A\in \mathcal{Q_\osf}}\ualgshift p_A.$$ Consequently, we obtain that
$$I=\ualgshift x=\ualgshift x' \subseteq \sum \limits_{A\in \mathcal{Q_\osf}}\ualgshift p_A$$ and so
$$Soc(\ualgshift)\subseteq \sum\limits_{A \in \mathcal{Q_\osf}}\ualgshift p_A.$$

As we mentioned at the beginning of the section, the other inclusion follows from Proposition~~\ref{prop2}, as it gives us that $\ualgshift p_A$ is a left minimal ideal of $\ualgshift$ for each $A\in \mathcal{Q_\osf}$. 

So, $$Soc(\ualgshift)=\sum\limits_{A \in \mathcal{Q_\osf}}\ualgshift p_A,$$ as desired.
\end{proof}

\begin{corollary}\label{socle via line paths} Let $\osf$ be a subshift and $R$ be a field. Then $Soc(\ualgshift)=J$, where $J$ is the two-sided ideal of $\ualgshift$ generated by the set $\{p_A:A\in \mathcal{Q_\osf}\}$. Moreover, the socle is a $\Z$-graded ideal.
\end{corollary}

\begin{proof} The first part of the result follows from Theorem~\ref{socle description} and the fact that $Soc(\ualgshift)$ is a two-sided ideal. The grading statement follows from \cite[Remark~2.1.2]{thebook}.
\end{proof}

\begin{remark}\label{carregador}
Let $E$ be a graph. By \cite[Proposition~5.1]{Socle1}, the socle of the Leavitt path algebra $L_K(E)$ contains (but is not necessarily equal to) the sum of the minimal left ideals associated with the line points. In Theorem~\ref{socle description} we have an equality, since we are summing over a larger set (notice that a line path in $E$ induces an irrational path in the subshift associated with $E$, but the converse is not necessarily true, as can be seen in Example \ref{irrational paths are greater than line paths}).
\end{remark} 


\begin{corollary}\label{iogurtecaseiro}  Let $\osf$ be a subshift and $R$ be a field. Then:
\begin{enumerate}
\item $\ualgshift$ has nonzero socle if, and only if, $\mathcal{Q_\osf}$ is non-empty.
\item $Soc(\ualgshift)\neq \ualgshift$.
\end{enumerate}
\end{corollary}

\begin{proof}
 The first statement is a direct consequence of Corollary~\ref{socle via line paths}, so we prove the second item.  The result is clearly true if $\mathcal{Q_\osf}=\emptyset$ and hence we may assume that $\mathcal{Q_\osf}\neq \emptyset$. 

Seeking for a contradiction, suppose that $p_\osf\in Soc(\ualgshift)$. From Theorem~\ref{socle description}, we obtain that $p_\osf$ is a sum of the form $$p_\osf=\sum\limits_{j=1}^m z_j p_{D_j},$$ where $z_j \in \ualgshift$ and $D_j\in \mathcal{Q_\osf}$ for each $j$.
By Lemma~\ref{many elements in Q}, $\mathcal{Q_\osf}$ contains infinitely many elements. So, there exists $D\in \mathcal{Q_\osf}$ such that $D\neq D_j$ for each $j\in \{1,...,m\}$, which implies that $p_{D_j}p_D=0$ for each $j$. Therefore,
$$p_D=p_\osf p_D=\sum\limits_{j=1}^m z_j p_{D_j}p_D=0,$$ which is a contradiction, since $p_D\neq 0$.
We conclude that $p_\osf\notin Soc(\ualgshift)$ and, consequently, $\ualgshift\neq Soc(\ualgshift)$. 
 \end{proof}

\begin{remark}
For finite graphs, the socle of a Leavitt path algebra is non-zero if, and only if, the graph has sinks, see \cite{Socle}. So, for a finite graph without sinks $E$, the socle of $L_K(E)$ is zero. We can obtain the last statement using the result above, considering the subshift $\osf$ associated with $E$, recalling that $\ualgshift$ is isomorphic to $L_K(E)$, and noticing that $\mathcal{Q_\osf}=\emptyset$. 
\end{remark}

In the remainder of this section, we focus on finding a smaller generating set for the socle of a subshift algebra. For this, we need the following definition.

\begin{definition}\label{def equiv line paths} Let $\osf$ be a subshift and $x,y\in \osf$. We say that the elements $x$ and $y$ are equivalent, and write $x\sim y$, if there exists $m,n\in \N$ such that $\sigma^m(x)=\sigma^{n}(y)$.
\end{definition}


In the proposition below, we denote by $\langle x \rangle$ the two-sided ideal generated by $x$ in $\ualgshift$.

\begin{proposition}\label{equivalent paths generate the same ideals} Let $\osf$ be a subshift and let $A,B\in \TCB$ be two single point sets, say, $A=\{p\}$ and $B=\{q\}$, where $p,q\in \osf$. If $p\sim q$ then $\langle p_A \rangle =\langle p_B \rangle$, and if $p\not \sim q$ then $\langle p_A \rangle \langle p_B \rangle =0$.
\end{proposition}

\begin{proof} 
We begin the proof showing that $p\sim q$ implies $\langle p_A \rangle=\langle p_B \rangle$.

Suppose first that $p=\sigma^n(q)$ for some $n\in \N$. Then $q=\alpha p$, where $\alpha\in \lang$. 
Hence, 
$$s_\alpha^*p_B s_\alpha=s_\alpha^* s_\alpha p_{r(B,\alpha)}=p_{F_\alpha} p_{r(B,\alpha)}=p_{F_\alpha\cap r(B,\alpha)}=p_A,$$ 
which implies that $\langle p_A\rangle\subseteq \langle p_B \rangle$. Multiplying the above equality on the left by $s_\alpha$, and on the right by $s_\alpha^*$, we get that
$$s_\alpha p_A s_\alpha^*=s_\alpha s_\alpha^* p_B s_\alpha s_\alpha^*=p_{Z_\alpha} p_B p_{Z_\alpha}=p_B,$$ and consequently $\langle p_B\rangle\subseteq \langle p_A \rangle$. Therefore, we have proved that $\langle p_A\rangle =\langle p_B \rangle$.

To prove the general case, suppose that $p\sim q$. Then, there exists $m,n\in \N$ such that $\sigma^m(p)=\sigma^n(q)$, and hence there exists $a,b\in \lang$, and $z\in \osf$, such that $p=a z$ and $q= b z$. Using that $r(A,\alpha)=\{z\}=r(B,\beta)$, and what we proved in the previous paragraph, we obtain that $\langle p_{A} \rangle = \langle p_{\{z\}} \rangle =\langle p_B \rangle$.

Now, suppose that $p$ and $q$ are not equivalent. We will show that $\langle p_{A} \rangle \langle p_B \rangle=0$. 

From Lemma \ref{parouAchuva!}, $\langle p_A \rangle$ is the linear span of elements of the form $s_\alpha s_\beta^* p_{\{\mu\}}$, where $\mu\sim p$ and $\alpha, \beta \in \lang$, and also $\langle p_B\rangle$ is the linear span of elements of the form $s_a s_b^* p_{\{\nu\}}$, where $\nu \sim q$ and $a,b \in \lang$. Therefore, it is enough to show that $s_\alpha s_\beta ^* p_{\{\mu\}}s_as_b^* p_{\{\nu\}}=0$, for $\alpha, \beta, a,b, \mu$ and $\nu$ as above.

For $a,b, \alpha, \beta, \mu, \nu$ as above, notice that, by the second item of Lemma~\ref{parouAchuva!} either $p_{\{\mu\}}s_a s_b^* =0$ or there exists $z\sim \mu$ such that 
$p_{\{\mu\}}s_a s_b^*= s_a s_b^* p_{\{z\}}$. As $p\not \sim q$, we have that $z \not \sim \nu$ and hence $p_{\{z\}}p_{\{\nu\}}=0$. Therefore, 
$$s_\alpha s_\beta ^* p_{\{\mu\}}s_as_b^* p_{\{\nu\}}=s_\alpha s_\beta^* s_a s_b^* p_{\{z\}}p_{\{\nu\}}=0$$ and, consequently, $\langle p_A \rangle \langle p_B \rangle=0$.

\end{proof}

Let $\mathcal{Q_\osf}/\hspace{-0.15 cm}\sim$ be the quotient space of $\mathcal{Q_\osf}$ by $\sim$, whose elements we denote by $[q]$ (here, and when necessary below, we identify a one-point set in $\mathcal{Q_\osf}$ with its element). To each equivalence class $[q]$ in $\mathcal{Q_\osf}/\hspace{-0.15 cm}\sim$ we associated a projection, denoted by $p_{[q]}$, which is defined as $p_{[q]}:= p_{\{q'\}}$, where $q'\in [q]$. By the proposition above, $p_{[q]}$ is well-defined. Joining this discussion with Corollary~\ref{socle via line paths} and Proposition~\ref{equivalent paths generate the same ideals},
we get the following description of the socle of a subshift algebra.

\begin{corollary}\label{socle characterization} Let $\osf$ be a subshift and $R$ be a field. Then, 
$$Soc(\ualgshift)=\bigoplus\limits_{[q]\in\mathcal{Q}_\osf/\sim } \langle [p_q] \rangle.$$

\end{corollary}

\section{The socle of a subshift algebra as a Leavitt path algebra.}\label{socle as leavitt algebra}

In this section, starting from a subshift $X$, we build a graph such that the associated Leavitt path algebra is graded isomorphic to the socle of the subshift algebra. Furthermore, we provide applications for such construction. The definition of the graph depends on some subsets of $X$, which we define below.

Fix an element $p\in \osf$ such that $\{p\}\in \mathcal{Q_\osf}$, and define $I_0=\N$. Moreover, let $$J_0=\{\sigma^{n}(p): n\in I_0\},$$ 
$$J_1=\sigma^{-1}(J_0)\setminus J_0,$$
and inductively define, for each $n\in \N$,
$$J_n=\sigma^{-1}(J_{n-1})\setminus (J_0\cup J_1\cup\ldots \cup J_{n-1}).$$ 
Notice that $J_n\cap J_m=\emptyset$ for each $n\neq m$. Moreover, it is possible that $J_n=\emptyset$ for some $n$, in which case this process is finite. 

For each $i\in I_0$ define $z_i=\sigma^i(p)$, so that $J_0=\{z_i:i\in I_0\}$. Inductively, for each $n\geq 1$,  chose an index set $I_n$ that indexes the elements of $J_n$ and such that $I_n \cap I_m=\emptyset$ for each $n\neq m$. Write $J_n=\{z_i:i\in I_n\}$.

Define $I=\bigcup\limits_{n=0}^\infty I_n$ and $J=\bigcup\limits_{n=0}^\infty J_n$, which are both disjoint unions. The set $I$ is the index set of the elements of $J$, so the map $I\ni i \mapsto z_i \in J$ is a bijective map.

\begin{remark}\label{voltouovento} Each element $z_i \in J$ has a unique representation of the form $z_i=az_j$, where $a \in \alf$ and $z_j \in J$. This is clear when $i\in I_0$, since in this case $z_i=\sigma^{i}(p)=a \sigma^{i+1}(p)$, where $a$ is the first letter of $z_i$. When $z_i\in J_n$ for some $n\geq 1$, we have that $z_i\in \sigma^{-1}(J_{n-1})$ and hence $\sigma(z_i)=z_j$ for a (unique) $z_j \in J_{n-1}$. Therefore, $z_i=az_j$ for some $a\in \alf$.
\end{remark}

Now, we define a graph $E_p$ associated to the fixed element $p\in \osf$ (with $\{p\}\in \mathcal{Q_\osf}$). Both the vertex and edge sets are indexed by the set $I$ defined above and we write $E^0=\{v_i:i\in I\}$ and $E^1=\{e_i:i\in I\}$. The source map is defined by $s(e_i)=v_i$, for each $i\in I$. To define the range map, for each $i\in I$, write $z_i=az_j$ as in Remark~\ref{voltouovento}, and let $r(e_i)=v_j$. Equivalently,  the range map is defined as $r(e_i)=v_j$, where $\sigma(z_i)=z_j$. Therefore, we get that $r(e_i)=v_j=s(z_j)=s(\sigma(z_i))$.



\begin{remark}\label{no closed paths in E_p} In the graph $E_p$ there are no closed paths, that is, strings $\alpha=e_{k_1}\ldots e_{k_n}$ such that $r(e_{k_l})=s(e_{k_{l+1}})$, $l=1\ldots n-1$, and $r(e_{k_n})=s(e_{k_1})$. Indeed, for each $n\geq 1$ and $i\in I_n$, we have that $s(e_i)=v_i$ and $r(e_i)=v_j$, where $j\in I_{n-1}$. So, if $\alpha$ is a path in the graph $E_p$ such that $s(\alpha)=v_i$, with $i\in I_n$ and $n\geq 1$, then the range of $\alpha$ lies in some $I_j$ with $j<n$. Hence $\alpha$ is not a closed path. Suppose now that $\alpha$ is a path such that $s(\alpha)=v_i$ for some $i\in I_0$. Notice that for $i\in I_0$, the source and range maps are defined as $s(e_i)=v_i$ and $r(e_i)=v_{i+1}$. Hence, $\alpha$ is not a closed path. Consequently, there are no closed paths in $E_p$.
\end{remark}

Our next goal is to prove that the ideal generated by an irrational element $\{p\}$ is isomorphic to the Leavitt path algebra of the graph $E_p$ defined above. Before we do this, we prove the following lemma.

\begin{lemma}\label{ventonocampeche} Let $\osf$ be a subshift and $Y\subseteq \osf$ be such that $\{y\}\in \TCB$ for each $y\in Y$. Then:
\begin{enumerate}
\item for each $n\in \N$ and $z\in \sigma^n(Y)$ it holds that $\{z\}\in \TCB$, and,
\item for each $n\in \N$ and $z\in \sigma^{-n}(Y)$ it holds that $\{z\}\in \TCB$.
\end{enumerate}
\end{lemma}

\begin{proof} Let $n\in \N$ and $z\in \sigma^n(Y)$. Then, $z=\sigma^{n}(y)$ for some $y\in Y$ and hence $y=\alpha z$, for some $\alpha \in \lang$. Since $\{y\}\in \TCB$, we have that $\{z\}=r(\{y\}, \alpha)\in \TCB$.

To prove the second item, let $n\in \N$ and $z\in \sigma^{-n}(Y)$. In this case, $\sigma^n(z)=y$ for some $y\in Y$, and therefore $z=\beta y$ for some $\beta \in \lang$. Since $\{y\} \in \TCB$, from the first item of Lemma~\ref{parouAchuva!}, we get that $\{z\}\in \TCB$.
\end{proof}

As a consequence of the above lemma, we obtain that every single element set formed by an element in the equivalence class of an irrational path is in $\TCB$, as stated below.

\begin{proposition}\label{quasesolnocampeche}
    Let $\{p\}\in \mathcal{Q_\osf}$. Then, $[p]=J$, where $J$ is the set constructed above. Moreover, for each $z\in [p]$, the set $\{z\}$ belongs to the boolean algebra $\TCB$.
\end{proposition}
\begin{proof}
   The first statement follows by the definition of $J$. For the second statement, notice that, by the first item of Lemma~\ref{ventonocampeche}, we have that $\{z\} \in \TCB$ for each $z\in J_0$. Also, by applying successively the second item of Lemma~\ref{ventonocampeche}, we get that $\{z\}\in \TCB$ for each $z\in J_n$, $n\geq 1$. 
\end{proof}

The following lemma will be useful in the study of Condition~(Y) within the graph associated with the socle of a subshift algebra.

\begin{lemma}\label{lema bijecoes Ep}let $\osf$ be a subshift,  $p\in \osf$ be such that $\{p\}\in \mathcal{Q_\osf}$, and $E_p$ be the associated graph. Denote by $E_p^1$ the edge set of $E_p$ and by $E_p^\infty$ the set of all the infinite paths in $E_p$. There exists bijections $\psi:[p]\rightarrow E_p^\infty$ and $\varphi:[p]\rightarrow E_p^1$ such that:
\begin{enumerate}
    \item For each $z\in [p]$, it holds that $\psi(z)=\varphi(z)\psi(\sigma(z))$.
    \item For each $z\in [p]$ and each $n\in \N$ it holds that
    $$\psi(z)=\varphi(z)\varphi(\sigma(z))\varphi(\sigma^2(z))...\varphi(\sigma^n(z))\psi(\sigma^{n+1}(z)).$$
    
    \item For each $z\in [p]$, it holds that $$\psi(z)=\varphi(z)\varphi(\sigma(z))\varphi(\sigma^2(z))\varphi(\sigma^3(z))...,$$ that is, the edge in position $n$  of the infinite path $\psi(z)$ is $\varphi(\sigma^{n}(z))$, for each $n\in \N$. 
\end{enumerate}
\end{lemma}

\begin{proof} We begin defining the maps $\psi$ and $\varphi$.
Let $I$ and $J$ be as at the beginning of Section~\ref{socle as leavitt algebra}. Recall that $I$ is the index set of $J$, so that $J=\{z_i:i\in I\}$ and $E_p^1=\{e_i:i\in I\}$. By Proposition~\ref{quasesolnocampeche}, we have that $[p]=J$ and hence $[p]=\{z_i:i\in I\}$. So, the map $\varphi:[p]\rightarrow E_p^1$ defined by $\varphi(z_i)=e_i$ is a bijection. Now we define $\psi$. From the definition of the source and range maps of $E_p$ (see Remark~\ref{voltouovento}) we get that $s(\varphi(z_i))=v_i$ for each $z_i \in [p]$ and  $r(\varphi(z_i))=s(\varphi(\sigma(z_i)))$. Moreover, from the definition of $E_p$, we have that each vertex $v_i$ of $E_p$ emits a unique edge, which is $e_i$, and so $v_i$ is the source of a unique infinite path in $E_p$ (the path beginning with $e_i$). Hence, we get a bijective map $\psi:[p]\rightarrow E_p^\infty$, defined by $\psi(z_i)=\xi_i$, where $\xi_i$ is the unique infinite path in $E_p$ with $s(\xi_i)=v_i$ (or, equivalently, $\xi_i$ is the unique infinite path in $E_p$ beginning with $e_i$). 

Now we prove the first item. Let $z_i\in [p]$, so that  $\varphi(z_i)=e_i$. Denote $\xi=\psi(z_i)$, that is, $\xi$ is the unique infinite path in $E_p$ beginning in the edge $e_i$. Write $z_i=a_i z_j$, where $a_i\in \alf$ and $z_j\in \osf$. Notice that, from the definition of $E_p$, we have that $r(e_i)=v_j$ and $v_j$ is the source of a unique edge, which is $e_j$. Therefore, the second edge of $\xi$ is $e_j$. Write $\xi=e_ie_j \widetilde{\xi}$. From the definition of $\psi$, we get that $\psi(z_j)$ is the unique infinite path of $E_p$ beginning with $e_j$ and, since $e_j\widetilde{\xi}$ is an infinite path beginning with $e_j$, we conclude that  $\psi(z_j)=e_j\widetilde{\xi}$. Then,
$$\psi(z_i)=\xi=e_ie_j\widetilde{\xi}=\varphi(z_i)\psi(z_j)=\varphi(z_i)\psi(\sigma(z_i)),$$ which proves the first item. 

The second and third items follow by applying successively the first one. 
\end{proof}

Notice that, by \cite[Remark~2.1.2]{thebook}, the ideal generated by a projection associated with an irrational element $\{p\}$ is graded. Next, we prove that this ideal is $\Z-$graded isomorphic to the Leavitt path algebra of the graph $E_p$. We refer the reader to \cite{thebook} for the concepts regarding Leavitt path algebras.

\begin{proposition}\label{graph and socle} Let $\osf$ be a subshift, let $D=\{p\}\in \mathcal{Q}_\osf$, and let $E_p$ be the associated graph as above. Then, $\langle p_D \rangle$ and the Leavitt path algebra $L_R(E_p)$ are $\Z$-graded isomorphic (with its natural $\mathbb{Z}$-gradings).
\end{proposition}

\begin{proof} To obtain a homomorphism $\varphi:L_R(E_p)\rightarrow \ualgshift$, we use the universal property of $L_R(E_p)$. So, it is enough to define $\varphi$ on $E^0\cup E^1$ in a way such that the images of $E^0\cup E^1$ satisfy the relations defining $L_R(E_p)$.

We begin with the definition of $\varphi$ on $E^0$. Let $v_i \in E^0$, where $i\in I$, and let $z_i$ be the associated element in $J$. Define $\varphi(v_i)=p_{\{z_i\}}$. Notice that, from Proposition~\ref{quasesolnocampeche}, we have $\{z_i\}\in \TCB$ and hence $\varphi(v_i)$ is well defined. Next, for each $e_i \in E^1$, let $z_i \in J$ be the element corresponding to $i$,  and let $a_i \in \alf$ be the first letter of $z_i$. Define $\varphi(e_i)=p_{\{z_i\}} s_{a_i}$, and $\varphi(e_i^*)=s_{a_i}^* p_{\{z_i\}}$.

Now we verify that $\{\varphi(v_i), \varphi(e_i), \varphi(e_i^*):i\in I\}$ satisfy the relations which define $L_R(E_p)$, see \cite[Definition~1.2.3]{thebook}.
First notice that $\varphi(v_i)$ is idempotent for each $v_i\in E^0$. Moreover, for $i\neq j$, we have that $\varphi(v_i)\varphi(v_j)=p_{\{z_i\}}p_{\{z_j\}}=0$, since $z_i \neq z_j$. 
Next, for a fixed $i\in I$, observe that
$$\varphi(s(e_i))\varphi(e_i)=\varphi(v_i)\varphi(e_i)=p_{\{z_i\}}p_{\{z_i\}} s_{a_i}=p_{\{z_i\}} s_{a_i}=\varphi(e_i).$$ Moreover, for this $i$, let $z_i \in J$ be the  element associated with $i$ and write $z_i=a_i z_j$, where $a_i$ is the first letter of $z_i$. Recall that $r(e_i)=v_j$. Then, $r(\{z_i\}, a_i)=\{z_j\}$ and hence
$$\varphi(e_i) \varphi(r(e_i))=\varphi(e_i)\varphi(v_j)=p_{\{z_i\}} s_{a_i}p_{\{z_j\}}=s_{a_i}p_{r(\{z_i\}, a_i)}p_{\{z_j\}}=p_{\{z_i\}}s_{a_i}=\varphi(e_i).$$ 

So, we proved that
$$\varphi(s(e_i))\varphi(e_i)=\varphi(e_i)=\varphi(e_i)\varphi(r(e_i)).$$
Similarly, the reader can check that $\varphi(e_i^*)=\varphi(e_i^*) \varphi(s(e_i))=\varphi(e_i^*)=\varphi(r(e_i))\varphi(e_i^*)$.

Next, for $i,j\in I$ with $i\neq j$, observe that 
$$\varphi(e_i^*)\varphi(e_j)=s_{a_i}^* p_{\{z_i\}}p_{\{z_j\}}s_{a_j}=0$$ since $p_{\{z_i\}}p_{\{z_j\}}=0$. Moreover, to conclude that $\varphi(e_i^*)\varphi(e_j)=\delta_{i,j}\varphi(r(e_i))$, write $z_i=a_i z_j$ as in Proposition~\ref{quasesolnocampeche} and notice that 
$$\varphi(e_i^*)\varphi(e_i)=s_{a_i}^*p_{\{z_i\}}s_{a_i}=s_{a_i}^*s_{a_i}p_{r(\{z_i\}, a_i)}=p_{F_{a_i}}p_{\{z_j\}}=p_{\{z_j\}}=\varphi(r(e_i)).$$

Finally, for each $i\in I$ (write $z_i=a_i z_j$ as in Proposition~\ref{quasesolnocampeche}), notice that $s^{-1}(v_i)=e_i$, and so 
 $$\sum\limits_{e\in s^{-1}(v_i)} \varphi(s_e)\varphi(s_e^*)=\varphi(e_i)\varphi(e_i^*)=p_{\{z_i\}}s_{a_i} s_{a_i}^* p_{\{z_i\}}=p_{\{z_i\}}p_{Z_{a_i}} p_{\{z_i\}}=p_{\{z_i\}}=p_{v_i}.$$
 
 We have checked that the image of $\varphi$ satisfies all the relations defining $L_R(E_p)$. By the universal property of $L_R(E_p)$, we have that $\varphi$ extends to a homomorphism, which we also call $\varphi:L_R(E_p)\mapsto \ualgshift$. By the definition of $\varphi$ on the generators, we obtain that it is a graded homomorphism.

To finish our proof, it remains to show that $\varphi$ is injective and that $\varphi(L_R(E_p))=\langle p_D \rangle$.

By Remark~\ref{no closed paths in E_p}, $E_p$ has no closed paths, and hence it satisfies (vacuously) Condition~$(L)$. Since $\varphi(v_i)=p_{\{z_i\}}\neq 0$ for each $v_i \in E^0$, we obtain from Cuntz-Krieger Uniqueness Theorem, see \cite[Theorem~2.2.15]{thebook}, that $\varphi$ is injective.

It remains to show that $\varphi(L_R(E_p))=\langle p_D \rangle$.
From Proposition \ref{quasesolnocampeche}, each $z_i\in J$ is an element of $[p]$, and then from Proposition \ref{equivalent paths generate the same ideals} we get that $p_{\{z_i\}}\in \langle p_D \rangle$. Therefore $\varphi(e_i), \varphi(e_i^*)$ and $\varphi(v_i)$ are all elements of $\langle p_D \rangle$, for all $i\in I$, and consequently $\varphi(L_R(E_p))\subseteq \langle p_D \rangle$.

To prove that $\langle p_D \rangle\subseteq \varphi(L_R(E_p))$ we first show the following claim.

{\it Claim: Let $\alpha z\in [p]$, where $\alpha \in \lang$ and $z\in \osf$. Then, $p_{\{\alpha z\}} s_\alpha$ and $s_\alpha^* p_{\{\alpha z\}}$ both belong to $\varphi(L_R(E_p))$.}

We show that $p_{\{\alpha z\}}s_\alpha \in \varphi(L_R(E_p))$ and leave the proof that $s_\alpha^*p_{\{\alpha z\}} \in \varphi(L_R(E_p))$, which is analogous, to the reader. 

Since $\alpha z\in [p]=J$, we have that  $\alpha z=z_{i_1}$ for some $i_1 \in I$. From the definition of $\varphi$, we obtain that $\varphi(e_{i_1})=p_{\{z_{i_1}\}}s_{\alpha_1}$, where $\alpha_1$ is the first letter of $\alpha$. Notice that $\sigma(\alpha z) \in J$, so that $\sigma(\alpha z)=z_{i_2}$ for some $i_2 \in I$. From the definition of $\varphi$, we get that $\varphi(e_{i_2})=p_{\{z_{i_2}\}}s_{\alpha_2}$. Proceeding inductively, we obtain indexes $i_k\in I$, with $k\in \{1,2,...,|\alpha|\}$, such that $\varphi(e_{i_k})=p_{\{z_{i_k}\}}s_{\alpha_{i_k}}.$

From the second item of Lemma \ref{universal properties} we get that $s_{\alpha_k}p_{\{z_{i_{k+1}}\}}=p_{\{z_{i_k}\}}s_{\alpha_k}$, for each $k\in \{1,...,|\alpha|-1\}$. Hence, making the proper computations, we obtain that $$\varphi(e_{i_1})\varphi(e_{i_2})...\varphi(e_{i_{|\alpha|}})=p_{\{\alpha z\}}s_{\alpha_1}...s_{\alpha_{|\alpha|}}=p_{\{\alpha z\}} s_\alpha.$$ Therefore, $p_{\{\alpha z\}}s_\alpha \in \varphi(L_R(E_p))$ and the claim is proved.

Now we show that $\varphi(L_R(E_p))=\langle p_D \rangle.$

Let $\alpha, \beta \in \lang$ and $A\in \TCB$ be such that $p_D s_\alpha p_A s_\beta^*\neq 0$. From the second item of Lemma~\ref{parouAchuva!}, there exists $y\in \osf$ such that $p=\alpha y$, $\{\beta y\}\in \TCB$ and moreover $p_D s_\alpha p_A s_\beta^*=s_\alpha s_\beta^* p_{\{\beta y\}}$. Applying the second item of Lemma~\ref{universal properties}, we get that $p_{\{\alpha y\}}s_\alpha=s_\alpha p_{\{y\}}$ and $s_\beta^* p_{\{\beta y\}}=p_{\{y\}}s_\beta^*$, and therefore 
$$p_{\{\alpha y\}} s_\alpha s_\beta^* p_{\{\beta y\}}=s_\alpha s_\beta^* p_{\{\beta y\}}.$$ It follows from the claim proved above that the element $p_{\{\alpha y\}} s_\alpha s_\beta^* p_{\{\beta y\}}$ belongs to $\varphi(L_R(E_p))$, and so $p_D s_\alpha p_A s_\beta^*\in \varphi(L_R(E_p))$.

Similarly one shows that $s_\alpha p_A s_\beta^* p_D\in \varphi(L_R(E_p))$, for all $\alpha, \beta \in \lang$ and $A\in \TCB$. Consequently, we obtain that $\langle p_D \rangle \subseteq \varphi(L_R(E_p))$ and this finishes the proof of the proposition.
\end{proof}

\begin{corollary}\label{doubleskiff} Let $\osf$ be a subshift. Then, there exists a graph $E$ such that $L_R(E)$ and $Soc(\ualgshift)$ are $\mathbb{Z}$-graded isomorphic (with its natural $\mathbb{Z}$-gradings).    
\end{corollary}

\begin{proof} 
For each $[p] \in \mathcal{Q}_\osf/\hspace{-0.15 cm}\sim$, let $E_p$ represent the corresponding graph (ensuring that $E_p$ and $E_q$ are disconnected for distinct equivalence classes $[p]$ and $[q]$).
Define $E$ as the disjoint union of all such graphs $E_p$, and observe that $$L_R(E)=\bigoplus\limits_{[p]\in\mathcal{Q}_\osf/\sim } L_R(E_p).$$

For each $[p]\in \mathcal{Q_\osf}/\hspace{-0.15 cm}\sim$, Proposition~\ref{graph and socle} guarantees the existence of an isomorphism $\varphi_p:L_R(E_p)\rightarrow \langle [p] \rangle$.
From  Corollary~\ref{socle characterization}, we get that $Soc(\ualgshift)=\bigoplus\limits_{[p]\in\mathcal{Q}_\osf/\sim } \langle [p] \rangle$. So, $\varphi: L_R(E)\rightarrow Soc(\ualgshift)$ defined by $\varphi=\bigoplus\limits_{[p]\in\mathcal{Q}_\osf/\sim } \varphi_p$ is an isomorphism.
\end{proof}

\begin{remark}\label{rmk48}
    Notice that the graph associated with a subshift described in Corollary~\ref{doubleskiff} is always acyclic, row-finite, and sinkless. Since the graph is acyclic, Corollary~\ref{doubleskiff} and \cite{MR2630124} implies that the socle is locally $K$-matricial; that is, it is the direct union of subalgebras, each isomorphic to a finite direct sum of finite matrix rings over the field $R$.  
\end{remark}

Using the description of the socle of a subshift algebra as the Leavitt path algebra of the associated graph given above, we provide next a criteria to determine when the socle is strongly graded (see \cite{MR4448424} for a study of strongly $\Z$-graded Leavitt path algebras). This criteria depends on Condition~(Y) so, for the reader's convenience, we first recall this definition. 

\begin{definition} (see \cite{MR3938862, GRgrad})
    A graph $E$ satisfies Condition~(Y) if for each infinite path $e_1e_2e_3\ldots$ there exists a finite path $\alpha$ and $k\in \N$ such that $s(e_{k+1})=r(\alpha)$ and $|\alpha|=k+1$.
\end{definition}

\begin{proposition}\label{regata}
    Let $\osf$ be a subshift. 
    The socle of $\ualgshift$ is strongly $\Z$-graded (with its natural $\mathbb{Z}-$grading) if, and only if, the associated graph $E$ given in Corollary~\ref{doubleskiff} satisfies Condition~(Y).
\end{proposition}
\begin{proof}
    This follows from Corollary~\ref{doubleskiff}, Remark~\ref{rmk48}, and \cite[Theorem~4.2]{MR3938862} (or  \cite[Theorem~3.9]{GRgrad}).
    \end{proof}

\begin{remark}\label{condition Y in E} Let $\osf$ be a subshift and let $E$ be the graph as in Corollary \ref{doubleskiff}, that is, $E$ is the disjoint union of all the graphs $E_p$, where $[p]\in \mathcal{Q}_\osf/\hspace{-0.15 cm}\sim$. Then $E$ satisfies condition $(Y)$ if, and only if, each $E_p$ satisfies condition $(Y)$.
\end{remark}

 Next, we characterize the condition $(Y)$ of $E$ in terms of the shift map $\sigma$ of $\osf$. Let $W\subseteq \osf$ be the set $W=\{p\in \osf: \{p\}\in \mathcal{Q}_\osf\}$. From the first item of Lemma \ref{ventonocampeche} we get that $\sigma(W)\subseteq W$, and from the second one we get that each $z\in \sigma^{-1}(x)$ is also an element of $W$, for each $x\in W$.
\begin{proposition}
    Let $\osf$ be a subshift and let $E$ be the associated graph as in Corollary~\ref{doubleskiff}. Then, $E$ satisfies condition $(Y)$ if and only if for each $q\in W$ there exists an $n\in \N$ such that $\sigma^{-(n+1)}(\sigma^n(q))\neq \emptyset$.
\end{proposition}

\begin{proof} First, suppose that for each $q\in W$ there exists an $n\in \N$ such that $\sigma^{-(n+1)}(\sigma^n(q))\neq \emptyset$. By Remark~\ref{condition Y in E}, it is enough to prove that, for each $[p]\in \mathcal{Q}_\osf/\hspace{-0.15 cm}\sim$, the graph $E_p$ satisfies condition $(Y)$. Fix a $[p]\in \mathcal{Q}_\osf/\hspace{-0.15 cm}\sim$. Let $\xi$ be an infinite path in the graph $E_p$, $\psi$ be the map defined in Lemma~\ref{lema bijecoes Ep}, and $z\in [p]$ be such that $\psi(z)=\xi$. Furthermore, let $m\in \N$ be such that $\sigma^{-(m+1)}(\sigma^m(z))\neq \emptyset$, and chose $y\in \sigma^{-(m+1)}(\sigma^m(z))$. Then, $\sigma^{m+1}(y)=\sigma^m(z)$. Denote this last element by $x$. So, we get that $y=a x$ and $z=b x$, where $a,b\in \lang$ are such that $|a|=m+1$ and $|b|=m$. Now. let $\eta=\psi(y)$. Bu the second item of Lemma~\ref{lema bijecoes Ep} we get that 
$$\xi=\psi(z)=\varphi(z)\varphi(\sigma(z))...\varphi(\sigma^{m-1}(z))\psi(\sigma^{m}(z))$$ that is, $\xi$ has the form $\xi=\alpha\psi(\sigma^m(z))$, where $\alpha$ is the finite path $$\alpha=\varphi(z)\varphi(\sigma(z))\varphi(\sigma^2(z))...\varphi(\sigma^{m-1}(z)).$$
Similarly, $\eta$ has the form $\eta=\beta\psi(\sigma^{m+1}(y))$, where 
$$\beta=\varphi(y)\varphi(\sigma(y))\varphi(\sigma^2(y))...\varphi(\sigma^{m}(y)).$$ 
Since $\sigma^{m+1}(y)=\sigma^m(z)$, we obtain that $\eta=\beta \psi(\sigma^m(z))$. So, we have proved that $\xi=\alpha\psi(\sigma^m(z))$ and that there exists an infinite path $\eta$ of the form $\eta=\beta \psi(\sigma^m(z))$ with $|\beta|=|\alpha|+1$. This proves that $E_p$ satisfies condition $(Y)$.

For the converse, suppose that $E$ satisfies condition $(Y)$, and let $q\in W$. Let $\psi$ be as in Lemma~\ref{lema bijecoes Ep}, and define $\xi=\psi(q)$, which is an infinite path in $E_q$.  Given that $E$ satisfies Condition $(Y)$, it follows that $E_q$ also satisfies Condition $(Y)$. Consequently, there exists an infinite path $\eta$ in $E_q$ such that $\eta=\alpha \mu$ and $\xi=\beta \mu$, where $\mu$ is an infinite path in $E_q$ and $\alpha, \beta$ are finite paths in $E_q$ with $|\alpha|=|\beta|+1$. Let $y \in [q]$ be such that $\psi(y)=\eta$ and write $m=|\beta|$. 
By the second item of Lemma~\ref{lema bijecoes Ep}, we get that 
$$\beta\mu=\xi=\psi(q)=\varphi(q)\varphi(\sigma(q))\varphi(\sigma^2(q))...\varphi(\sigma^{m-1}(q))\psi(\sigma^m(q)),$$ from where we conclude that $\mu=\psi(\sigma^m(q))$. Similarly, from 
$$\alpha\mu=\eta=\psi(y)=\varphi(y)\varphi(\sigma(y))\varphi(\sigma^2(y))...\varphi(\sigma^m(y))\psi(\sigma^{m+1}(y)),$$ we obtain that $\mu=\psi(\sigma^{m+1}(y))$. Therefore $\psi(\sigma^{m+1}(y))=\mu=\psi(\sigma^m(q))$ and, since $\psi$ is injective, we have that $\sigma^{m+1}(y)=\sigma^m(q)$. This last equality means that $y\in\sigma^{-(m+1)}(\sigma^m(q))$, and so we are done.
\end{proof}

\begin{remark} If $\sigma_{|_W}:W\rightarrow W$ is surjective then, by the proposition above, $E$ satisfies condition $(Y)$. However, surjectivity of $\sigma$ is not necessary, as can be seen in the following example: consider the subshift $\osf$ over the alphabet $\alf=\{0,1,2\}$ determined by the set of forbidden words $F=\{00, 10, 20\}$. The associated graph satisfies Condition~$(Y)$ but $\sigma_{|_W}$ is not onto.
    
\end{remark}

As an application of Proposition~\ref{regata}, we show how to use the graded structure of the socle as an invariant of conjugacy between Ott-Tomforde-Willis subshifts.

\begin{theorem}\label{conjugacy} Let $\osf_1$ and $\osf_2$ be subshifts and $E_i$, $i=1,2$, be the associated graphs as in Corollary~\ref{doubleskiff}. If $E_1$ satisfies Condition~(Y) and $E_2$ does not, then the associated OTW-subshifts are not conjugate.
\end{theorem}

\begin{proof}
Suppose that the OTW subshifts are conjugate. Then, by \cite[Theorem~7.6]{BGGV}, there is a graded isomorphism between the associated subshift algebras, which induces a graded isomorphism between the socle of both algebras. By Proposition~\ref{regata}, the socle of one subshift algebra is strongly graded, while the other is not. Since a graded isomorphism preserves the strong graded structure of an algebra, we obtain a contradiction. 
\end{proof}

\begin{remark}\label{isometriconj}
    Under the hypothesis of the theorem above, using \cite[Theorem~6.11]{BGGV3}, we conclude that the associated subshifts built using the $\frac{1}{2^i}$ metric (which induces the product topology) are not isometrically conjugate.
        Moreover, for finite alphabets, OTW subshifts coincide with the usual notion of subshifts (with the product topology). Hence, the result above also provides an invariant for the conjugacy of subshifts (with the product topology) over finite alphabets. 
    \end{remark}

To illustrate an application of the theorem above, we provide two examples of subshifts. For one of them, Condition~(Y) is satisfied for the associated graph, while for the other it is not. We begin with a subshift that induces a 2-step Ott-Tomforde-Willis subshift not conjugate to any 1-step Ott-Tomforde-Willis subshift.  In this case, the graph does not satisfy Condition~(Y).

\begin{example}\label{example01} Let $\N^*$ be the alphabet and define $P$ as the following subset of $(\N^*)^3$: 
$$P:=\{(1,j,1),(j,1,j):j\in \N^*\setminus\{2\}\}\cup\{(j,j+1,j+2):j\in \N^*\}.$$ Let $F=(\N^*)^3\setminus P$ and $\osf_F:=\osf$. By [\cite{DDStep}, Section 3], the Ott-Tomforde-Willis subshift associated with $F$ is not conjugate to any 1-step Ott-Tomforde-Willis subshift. Notice that 
$$\osf:=\{ (1j)^\infty, (j1)^\infty: j\in \N^*\setminus\{2\}\}\cup \{(k+j)_{j\in \N^*}: k\in \N^* \}.$$

We show that $Soc(\ualgshift)$ is $R$-isomorphic to $M_\infty(R)$, the algebra of infinite matrices with finitely many nonzero entries.

The set of the irrational paths, in the subshift $\osf$, is $\mathcal{Q_\osf}=\{\{(k+j)_{j\in \N}\}:k\in \N^*\}$. Let $p=12345...$. Then, according to the equivalence relation given in Definition~\ref{def equiv line paths}, all the elements of $\mathcal{Q_\osf}$ are equivalent to $A=\{p\}$. By Corollary~\ref{socle characterization} we conclude that  $$Soc(\ualgshift)=\langle p_A\rangle.$$

For the element $p=1234...$, the associated graph $E_p$ is the graph below. 
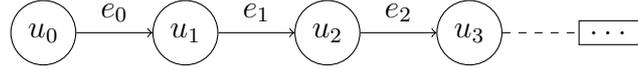
\begin{center}
\begin{tikzpicture}[node distance=1cm, auto]
  \node[draw,circle] (w0) {$u_0$};
  \node[draw,circle, right=of w0] (w1) {$u_1$};
  \node[draw,circle, right=of w1] (v1) {$u_2$};
  \node[draw,circle, right=of v1] (v2) {$u_3$};
  \node[draw, right=of v2] (dots1) {$\dots$};
  \draw[->] (w0) edge node{$e_0$} (w1);
  \draw[dashed] (v2) -- (dots1);
  \draw[->] (w1) edge node{$e_1$} (v1);
  \draw[->] (v1) edge node{$e_2$} (v2);
  \end{tikzpicture}
\end{center}

Given that $L_R(E_p)$ is isomorphic to $M_\infty(R)$ and, by  Proposition~\ref{graph and socle}  $L_R(E_p)$ and $\langle p_A \rangle$ are isomorphic, it follows that $Soc(\ualgshift)$ and $M_\infty(R)$ are isomorphic.

\end{example}

The next example exhibits a subshift such that the graph associated with the socle satisfies Condition~(Y). 

\begin{example}\label{example08} Let 
$$\osf:=\{(k+j)_{j\in \N^*}: k\in \Z \}.$$

The set of the irrational paths is $\mathcal{Q_\osf}=\{(k+j)_{j\in \N^*}:k\in \Z\}$. Let $p=12345...$. As before, all the elements of $\mathcal{Q_\osf}$ are equivalent to $A=\{p\}$ and hence, by Corollary~\ref{socle characterization}, we have that  $Soc(\ualgshift)=\langle p_A\rangle.$

To construct the graph associated with $\osf$, we chose $I_n=\{-n:n\in \N^*\}$ and $I=\Z$. The graph $E_p$ is depicted below. 
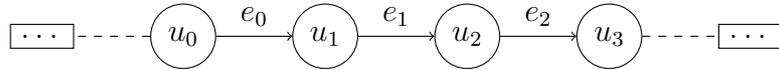
\begin{center}
\begin{tikzpicture}[node distance=1cm, auto]
  \node[draw,circle] (w0) {$u_0$};
  \node[draw,circle, right=of w0] (w1) {$u_1$};
  \node[draw,circle, right=of w1] (v1) {$u_2$};
  \node[draw,circle, right=of v1] (v2) {$u_3$};
  \node[draw, right=of v2] (dots1) {$\dots$};
  \draw[->] (w0) edge node{$e_0$} (w1);
  \draw[dashed] (v2) -- (dots1);
  \draw[->] (w1) edge node{$e_1$} (v1);
  \draw[->] (v1) edge node{$e_2$} (v2);
  \node[draw, left=of w0] (dots1) {$\dots$};
  \draw[dashed] (w0) -- (dots1);
  \end{tikzpicture}
\end{center}

The Leavitt path algebra of the graph above is isomorphic to $M_\infty(R)$ and hence so is $Soc(\ualgshift)$. Moreover, by Proposition~\ref{regata}, $Soc(\ualgshift)$ is strongly graded. 

\end{example}

Since one of the graphs of the examples above satisfies Condition~(Y) and the other does not, we obtain that the corresponding OTW-subshifts are not conjugate, as stated below. 

\begin{corollary}\label{saidarapida}
    The Ott-Tomforde-Willis subshifts associated with the subshifts of  Examples~\ref{example01} and \ref{example08} are not conjugate. Moreover, there is no isometric conjugacy between the associated subshifts built using the $\frac{1}{2^i}$ metric (which induces the product topology). 
\end{corollary}
\begin{proof}
This follows directly from Theorem~\ref{conjugacy} and Remark~\ref{isometriconj}.
\end{proof}

\section{The relation between the algebraic and the C*-socle}

We conclude the paper with a concise examination of the socle of the C*-algebra associated with a subshift, aiming to describe the relation between the socles in both analytical and purely algebraic contexts. In the analytical context, the socle is also defined as the sum of all left minimal ideals (it's important to note that these ideals are not necessarily closed). Below, we recall the definition of the C*-algebra associated with a subshift over an arbitrary alphabet, as presented in \cite{BGGV3}.

\begin{definition} Let $\osf$ be a subshift.
We define $\ucsalgshift$ as the universal unital C*-algebra generated by projections $\{p_A: A\in\TCB\}$ and partial isometries $\{s_a: a\in\alf\}$ subject to the relations:
\begin{enumerate}[(i)]
    \item $p_{\osf}=1$, $p_{A\cap B}=p_Ap_B$, $p_{A\cup B}=p_A+p_B-p_{A\cap B}$ and $p_{\emptyset}=0$, for every $A,B\in\TCB$;
    \item $s_{\beta}s^*_{\alpha}s_{\alpha}s^*_{\beta}=p_{C(\alpha,\beta)}$ for all $\alpha,\beta\in\lang$, where $s_{\eword}:=1$ and, for $\alpha=\alpha_1\ldots\alpha_n\in\lang$, $s_\alpha:=s_{\alpha_1}\cdots s_{\alpha_n}$ and $s_\alpha^*:=s_{\alpha_n}^*\cdots s_{\alpha_1}^*$.
\end{enumerate}
\end{definition}

\begin{remark}\label{ualgshift contained in csualgshift}
    It is shown in \cite{BGGV3} that the subshift algebra $\ualgshift$ is embedded densely in $\ucsalgshift$ via a homomorphism that sends generators of $\ualgshift$ to generators of $\ucalgshift$.  
\end{remark}

Next, we show that the socle of $\ualgshift$ is contained in the socle of $\ucalgshift$ and give an example where the inclusion is proper. We will use the concept of minimal idempotent recalled below Lemma~\ref{many elements in Q}.
Taking into consideration that $\ualgshift$ and $\ucalgshift$ are semiprime (see \cite[Corollary~5.6]{reductiontheoremofsubshifts} for primeness of $\ualgshift$), we have the following result.

\begin{lemma}\label{p_A are minimal idempotents}
    Let $\osf$ be a subshift, $R$ be a field and $A\in \mathcal{Q_\osf}$. Then, $p_A$ is a minimal idempotent in $\uCalgshift$ and also in $\ucsalgshift$.
\end{lemma}

\begin{proof} 
Consider an element $A\in \mathcal{Q}_\osf$. From the proof of Proposition~\ref{elements of Q generate minimal ideals}, it follows that $p_A \uCalgshift p_A=\mathbb{C} p_A$, implying that $p_A$ is a minimal idempotent in $\uCalgshift$. Furthermore, by Remark~\ref{ualgshift contained in csualgshift}, we deduce that the closure of $p_A \uCalgshift p_A$ in $\ucalgshift$, denoted by $\overline{p_A \uCalgshift p_A}$, equals $p_A \ucalgshift p_A$. However, as $p_A \ualgshift p_A=\mathbb{C} p_A$, we conclude that
$$\mathbb{C}p_A=\overline{\mathbb{C}p_A}=\overline{p_A \uCalgshift p_A}=p_A \ucalgshift p_A.$$ Hence, 
$p_A \ucalgshift p_A=\mathbb{C}p_A$ and therefore $p_A$ is also a minimal idempotent in $\ucalgshift$.    
\end{proof}

\begin{proposition}\label{contenidos} Let $\osf$ be a subshift. Then,
$$Soc(\uCalgshift)\subseteq Soc(\ucalgshift).$$ 
\end{proposition}

\begin{proof} According to Corollary~\ref{socle via line paths}, $Soc(\uCalgshift)$ is the two-sided ideal of $\uCalgshift$ generated by the set $\{p_A: A\in \mathcal{Q_\osf}\}$. To prove that  $Soc(\uCalgshift)\subseteq Soc(\ucalgshift)$, since $Soc(\ucalgshift)$ is also a two-sided ideal, it is enough to show that $p_A\in Soc(\ucalgshift)$ for each $A\in \mathcal{Q_\osf}$. By Lemma~\ref{p_A are minimal idempotents}, $p_A$ is a minimal idempotent element in $\ucalgshift$ for each $A\in \mathcal{Q_\osf}$. Therefore, from \cite[Proposition~30.6]{SocleBook},  we obtain that $\ucalgshift p_A$ is a minimal left ideal in $\ucalgshift$ and so, in particular, $p_A\in Soc(\ucalgshift)$ for each $p_A \in \mathcal{Q_\osf}$. Consequently, $Soc(\uCalgshift)\subseteq Soc(\ucalgshift)$.


      

\end{proof}

We conclude the paper with an example illustrating that the inclusion $Soc(\uCalgshift) \subseteq Soc(\ucalgshift)$ may be strict. 

\begin{example} 

Let $\osf$ be the subshift of Example \ref{example01}, that is, $$\osf:=\{ (1j)^\infty, (j1)^\infty: j\in \N^*\setminus\{2\}\}\cup \{(k+j)_{j\in \N^*}: k\in \N^* \}.$$ Let $p=1234...$ and $A=\{p\}$. 

The graph $E_p$ associated with the element $p$ is the graph 
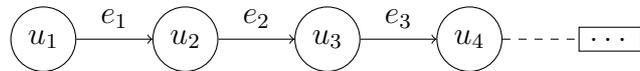
\begin{center}
\begin{tikzpicture}[node distance=1cm, auto]
  \node[draw,circle] (w0) {$u_1$};
  \node[draw,circle, right=of w0] (w1) {$u_2$};
  \node[draw,circle, right=of w1] (v1) {$u_3$};
  \node[draw,circle, right=of v1] (v2) {$u_4$};
  \node[draw, right=of v2] (dots1) {$\dots$};
  \draw[->] (w0) edge node{$e_1$} (w1);
  \draw[dashed] (v2) -- (dots1);
  \draw[->] (w1) edge node{$e_2$} (v1);
  \draw[->] (v1) edge node{$e_3$} (v2);
  \end{tikzpicture}
\end{center}
\end{example}

Let $H$ be a separable Hilbert space, with an orthonormal Schauder basis $\{\delta_n:n\geq 1\}$, and let $K(H)$ denote the C*-algebra of all the compact operators in $H$. For each $n\geq 1$, define $S_{e_n}, P_{u_n} \in K(H)$ as follows: $S_{e_n}(\delta_{n+1})=\delta_n$ and $S_{e_n}(\delta_k)=0$ for each $k\neq n+1$, and $P_{u_n}(\delta_n)=\delta_n$ and $P_{u_n}(\delta_k)=0$ for each $k\neq n$. Notice that the Hilbert adjoint  operator $S_{e_n}^*$ of $S_{e_n}$ is such that $S_{e_n}^*(\delta_n)=\delta_{n+1}$, and $S_{e_n}^*(\delta_k)=0$ for $k\neq n$.  

It is well known that there exits an injective homomorphism $\Phi:L_\mathbb{C}(E_p)\rightarrow K(H)$ such that $\Phi(e_n)=S_{e_n}$, $\Phi(e_n^*)=S_{e_n}^*$ and $\Phi(u_n)=P_{u_n}$. Moreover, this homomorphism extends to an isomorphism of C*-algebras $\Psi:C^*(E_p)\rightarrow K(H)$. 

Now, let $\varphi:L_\mathbb{C}(E_p)\rightarrow \uCalgshift$ be as in Proposition \ref{graph and socle} (that is, $\varphi$ is defined on the generators as follows: $\varphi(u_i)=p_{\{z_i\}}$, $\varphi(e_i)=p_{\{z_i\}}s_i$ and $\varphi(e_i^*)=s_i^*p_{\{z_i\}}$, where $z_i=i(i+1)(i+2)...\in \osf$). By Corollary~\ref{doubleskiff}, we have that $Im(\varphi)=Soc(\ualgshift)$. Let $B=\Phi(L_\mathbb{C}(E_p))$, so that $\varphi\circ \Phi^{-1}:B\rightarrow Soc(\ualgshift)$ is an isomorphism. 
Through direct calculations, we obtain that the matrix associated with each element of $B$, relative to the basis ${\delta_n : n \geq 1}$, belongs to $M_\infty(\mathbb{C})$. Furthermore, for every element $T \in M_\infty(\mathbb{C})$, there exists an element in $B$ whose associated matrix is precisely $T$. Consequently, there exists an isomorphism between $B$ and $M_\infty(\mathbb{C})$, implying that $Soc(\ualgshift)$ is isomorphic to $M_\infty(\mathbb{C})$.

Recall that $\varphi$ is injective and hence it extends to an (injective) homomorphism $\psi:C^*(E_p)\rightarrow \ucalgshift$. So, we get an injective homomorphism $f:K(H)\rightarrow \ucalgshift$ defined by $f=\psi\circ \Psi^{-1}$. Notice that $f(B)=Soc(\ualgshift)$.

Next, we show that $Soc(\ualgshift)$ is strictly contained in $Soc(\ucalgshift)$. To prove this, we show that there exists an element $L\in K(H)$ such that $f(L)\in Soc(\ucalgshift)$ and $L$ has the property that its associated matrix, relative to the basis $\{\delta_n:n\geq 1\}$, does not belong to $M_\infty(\mathbb{C})$. 

Define, for each $x,y\in H$,
the operator $L_{x,y} \in K(H)$ by $L_{x,y}(z)=x\langle z,y\rangle$.
Through a few straightforward computations, we obtain that  $L_{\delta_i,\delta_j}=\Psi(e_ie_{i+1}...e_{j-1})$ for each $i<j$, that $L_{\delta_i,\delta_i}=\Psi(u_i)$, and that $L_{\delta_i, \delta_j}=\Psi(e_{i-1}^*e_{i-2}^*...e_j^*)$ if $i>j$.

Notice that $f(L_{\delta_1,\delta_j})=p_{z_1}s_1...s_{j-1}$ for each $j>1$, and that $f(L_{\delta_1,\delta_1})=p_{\{z_1\}}$. So,  $f(L_{\delta_1,\delta_j})\in p_{\{z_1\}}\ucalgshift$ for each $j\geq 1$.
Now, let $h=\sum\limits_{j=1}^\infty{\frac{1}{j}}\delta_j$ and, for each $n\geq 1$, define $h_n=\sum\limits_{j=1}^n\frac{1}{j}\delta_j$. Observe that $f(L_{\delta_1, h_n})=\sum\limits_{j=1}^n \frac{1}{j}f(L_{\delta_1,\delta_j})\in p_{\{z_1\}}\ucalgshift$. Since $h_n \rightarrow h$ (in $H$), we obtain that $L_{\delta_1,h_n}\rightarrow L_{\delta_1,h}$ in $K(H)$, so that $f(L_{\delta_1,h})$ belongs to the closure of $p_{\{z_1\}}\ucalgshift$ (in $\ucalgshift$). As the set $p_{\{z_1\}}\ucalgshift$ is closed in $\ucalgshift$, we conclude that $f(L_{\delta_1,h})\in p_{\{z_1\}}\ucalgshift$. 

Now, by Lemma~\ref{p_A are minimal idempotents}, we get that $p_{\{z_1\}}$ is a minimal idempotent. Hence, from \cite[Proposition~30.6]{SocleBook}, we get that $p_{\{z_1\}} \in Soc(\ucalgshift)$. Since $Soc(\ucalgshift)$ is a two-sided ideal, we obtain that $p_{\{z_1\}}\ucalgshift\subseteq Soc(\ucalgshift)$, and thus $f(L_{\delta_1,h})\in Soc(\ucalgshift)$. 

To finish the proof, notice that the matrix of $L_{\delta_1, h}$, relative to the basis $\{\delta_n:n\geq 1\}$, has infinitely many nonzero elements, and so this matrix does not belong to $M_\infty(\mathbb{C})$. So, the element $L=L_{\delta_1,h}$ has the property that its associated matrix is not an element of $M_\infty(\mathbb{C})$, but $f(L) \in Soc(\ucalgshift)$. Therefore, $Soc(\ualgshift)$ (which is isomorphic to $M_\infty(\mathbb{C})$), is strictly contained in $Soc(\ucalgshift)$.

\bibliographystyle{abbrv}
\bibliography{ref}

\end{document}